\documentclass{article}
\usepackage[utf8]{inputenc}
\usepackage{amsfonts}
\usepackage{amsmath}
\usepackage{graphicx}
\usepackage{amssymb}
\usepackage{graphicx}
\usepackage{amsthm}
\usepackage{t1enc}
\usepackage{subfig}
\usepackage{MnSymbol}
\usepackage{mathrsfs}
\usepackage{bbm}
\usepackage{mathtools}
\usepackage[shortlabels]{enumitem}
\usepackage{tikz}
\usetikzlibrary{matrix}
\usepackage{tikz-cd}
\usepackage[cal=boondoxo]{mathalfa}
\usetikzlibrary{angles,quotes}

\usepackage{url}

\usepackage[a4paper ]{geometry}

\newtheorem{theorem}{Theorem}[section]
\newtheorem{lemma}[theorem]{Lemma}
\newtheorem{proposition}[theorem]{Proposition}
\newtheorem{corollary}[theorem]{Corollary}
\newtheorem{definition}[theorem]{Definition}

\newtheorem{conjecture}[theorem]{Conjecture}
\newtheorem{remark}[theorem]{Remark}
\newtheorem{statement}[theorem]{Statement}

\def\ord{\textup{\textrm{ord}}}
\def\N{\textup{\textrm{N}}}

\def\sign{\textup{\textrm{sign}}}

\def\char{\textup{\textrm{char}}}

\def\Gal{\textup{\textrm{Gal}}}
\def\Hom{\textup{\textrm{Hom}}}
\def\Log{\textup{\textrm{Log}}}

\def\Xint#1{\mathchoice
{\XXint\displaystyle\textstyle{#1}}{\XXint\textstyle\scriptstyle{#1}}%
{\XXint\scriptstyle\scriptscriptstyle{#1}}{\XXint\scriptscriptstyle\scriptscriptstyle{#1}}\!\int}

\def\XXint#1#2#3{{\setbox0=\hbox{$#1{#2#3}{\int}$}\vcenter{\hbox{$#2#3$}}\kern-.5\wd0}}

\def\multint{\Xint\times}

\usepackage[colorlinks,
pdfpagelabels,
pdfstartview = FitH,
bookmarksopen = true,
bookmarksnumbered = true,
linkcolor = black,
plainpages = false,
hypertexnames = false,
citecolor = black, urlcolor=black]{hyperref}

\tikzset{mynode/.style={font=\footnotesize,inner sep=0pt,text=black}
}

\title{Comparing Two Formulas for the Gross-Stark Units}
\author{Matthew H. Honnor }
\date{\today}

\begin{document}

\maketitle

\begin{abstract}
Let $F$ be a totally real number field. Dasgupta conjectured an explicit $p$-adic analytic formula for the Gross-Stark units of $F$. In a later paper, Dasgupta-Spie\ss \ conjectured a cohomological formula for the principal minors and the characteristic polynomial of the Gross regulator matrix associated to a totally odd character of $F$. Dasgupta-Spie\ss \ conjectured that these conjectural formulas coincide for the diagonal entries of Gross regulator matrix. In this paper, we prove this conjecture when $F$ is a cubic field.
\end{abstract}

\tableofcontents

\section{Introduction}

Let $F$ be a number field of degree $n$ with ring of integers $\mathcal{O}=\mathcal{O}_F$. Let $\mathfrak{p}$ be a prime of $F$, lying above $p \in \mathbb{Q}$, and let $H$ be a finite abelian extension of $F$ such that $\mathfrak{p}$ splits completely in $H$. In 1981, Tate proposed the Brumer-Stark conjecture (Conjecture $5.4$, \cite{MR656067}), stating the existence of $\mathfrak{p}$-unit $u$ in $H$, the Gross-Stark unit. This unit has $\mathfrak{P}$ order equal to the value of a partial zeta function at 0 for a prime $\mathfrak{P}$ above $\mathfrak{p}$. Since the unit $u$ is only non-trivial when $F$ is totally real and $H$ is totally complex containing a complex multiplication (CM) subfield, we assume this for the remainder of the paper. Recent work of Dasgupta-Kakde in \cite{Brumerstark} has shown that the Brumer-Stark conjecture holds away from $2$.

We begin by studying a conjecture of Dasgupta-Spie\ss \ presented in \cite{MR3968788}. In (Conjecture 3.1, \cite{MR3968788}) Dasgupta-Spie\ss \ conjecture a cohomological formula for the principle minors and the characteristic polynomial of the Gross regulator matrix associated to a totally odd character of the totally real field $F$. The diagonal terms of the Gross regulator matrix are defined via the Gross-Stark units. Let $\chi$ be our chosen totally odd character. Then the diagonal terms are expressed via the ratio of the $p$-adic logarithm and the $\mathfrak{p}$-order of the $\chi^{-1}$ component of the Gross-Stark unit. By considering (Conjecture 3.1, \cite{MR3968788}) for the $1 \time 1$ principle minors, Dasgupa-Spie\ss \ conjecture a formula for this value.

In \cite{MR2420508}, Dasgupta constructed explicitly, in terms of the values of Shintani zeta functions at $s=0$, an element $u_D \in F_\mathfrak{p}^\ast$ (Definition 3.18, \cite{MR2420508}). In (Conjecture 3.21, \cite{MR2420508}), Dasgupta conjectured that this unit is equal to the image of the Gross-Stark unit inside $F_\mathfrak{p}^\ast$. This formula has recently been shown to be correct up to a root of unity by Dasgupta-Kakde in \cite{intgrossstark}. Since the diagonal terms of the Gross regulator matrix are defined via the Gross-Stark units, one can use Dasgupta's formula to conjecture a second formula for their values.

The main result of this paper (Theorem \ref{thmforneq3}) is that, when $F$ is a cubic field ($n=3$), Dasgupta's conjecture agrees with the conjecture of Dasgupta-Spie\ss . This result was conjectured by Dasgupata-Spie\ss \ for $F$ of any degree (Remark 4.5, \cite{MR3968788}) and they proved the case when $F$ is a quadratic field ($n=2$) (Theorem 4.4, \cite{MR3968788}). 

We note that our main result (Theorem \ref{thmforneq3}) has been attempted previously by Tsosie in \cite{tsosie2018compatibility}. However, as we show in the appendix, we find a counterexample to the statement of a lemma necessary for his proof (Lemma 2.1.3, \cite{tsosie2018compatibility}). The statement concerns having a nice translation property of Shintani sets, for more details see Statement \ref{STconj} in the appendix. The main contribution of this paper is the methods we have developed to recover some control of the translation properties of Shintani sets. This is done in $\S 6.2$.

Our main result (Theorem \ref{thmforneq3}) combined with (Theorem 1.6, \cite{intgrossstark}) of Dasgupta-Kakde allows us to make progress on another conjecture of Dasgupta-Spie\ss, in particular, (Conjecture 3.1, \cite{MR3968788}). We are able to show that, if $F$ is a cubic field ($n=3$), then (Conjecture 3.1, \cite{MR3968788}) holds for the $1 \times 1$ principle minors of the Gross regulator matrix. I.e., Dasgupta-Spie\ss's cohomological formula for diagonal entries of the Gross-Regulator matrix is correct.

\subsection{Summary of proof}

In this summary we will assume that $F$ is a cubic field ($n=3$). The strategy for the proof of our main result builds on the ideas of Tsosie in \cite{tsosie2018compatibility}. A key element of each of the constructions are Shintani sets. We define these in \S 3, but for this summary it is enough to think of them as subsets of $\mathbb{R}_+^3$. Note that we are able to embed $F$ into $\mathbb{R}_+^3$ via its real embeddings. Let $\mathfrak{f}$ be the conductor of $H/F$ and $E_+(\mathfrak{f})$ the group of totally positive units of $F$ which are congruent to $1$ modulo $\mathfrak{f}$. Since we have assumed $n=3$ we have that $E_+(\mathfrak{f})$ is free of rank 2. 

Each of the formulas require Shintani sets $\mathcal{D}$ which are fundamental domains for the action of $E_+(\mathfrak{f})$. We refer to such Shintani sets as Shintani domains. When trying to show the equality of the formulas it is possible to reduce to showing the equality of Shintani sets. The first problem we need to overcome is that we may be unable to choose generators of $E_+(\mathfrak{f})$ such that they satisfy a required sign property. This sign property is required to make the Shintani domains we use well defined. The central difficulty in proving our main theorem is that in general there is no bound on the translation of a Shintani domain by an element contained in it. Without such control over the translations we are unable to show the required equality of Shintani sets. In the appendix we give more details on this lack of control.

The first step in negotiating both these problems is that we work with a free finite index subgroup $V \subset E_+(\mathfrak{f})$ of rank 2, rather than with the full $E_+(\mathfrak{f})$. In particular, we show that if we can prove the formulas agree when using a Shintani set, say $\mathcal{D}_V$, that is a fundamental domain for the action of $V$ then the formulas agree with the original $\mathcal{D}$. We refer to such Shintani sets as Colmez domains.

Building on the work of Colmez in \cite{MR922806}, we choose a free subgroup and generators $\langle \varepsilon_1, \varepsilon_2 \rangle = V \subset E_+(\mathfrak{f})$ such that $\varepsilon_1, \varepsilon_2$ satisfy the required sign property. With further work and calculations, we show that it is possible to choose $\varepsilon_1, \varepsilon_2$ to satisfy some additional properties which will allow us to have control over the translates of $\mathcal{D}_V$. It is our methods to find conditions we can choose for $\varepsilon_1 , \varepsilon_2$ to gain control of the translates of $\mathcal{D}_V$ that is the novel idea of this paper.

We are then left to explicitly calculate the two formulas, when working with $V$. We will reduce the task of showing the equality of the formulas to showing the equality of a collection of Shintani sets. The translation properties that we show on $\mathcal{D}_V$ allow us to complete the proof.

\subsection{Acknowledgements}

The author would like to thank Mahesh Kakde for many stimulating discussions over the course of this research and for his valuable comments on earlier versions of this paper. He would also like to thank Samit Dasgupta for useful conversations, regarding in particular the counterexample given in the appendix.

The author wishes to acknowledge the financial support of the Engineering and Physical
Sciences Research Council [EP/R513064/1] and King's College London.

\section{The Gross-Stark units}

Let $R$ denote a finite set of places of $F$ such that $\mathfrak{p} \nin R$, $R$ contains the archimedean places and $R$ contains the places that are ramified in $H$. We write $R_\infty$ for the set of archimedean places of $F$ and let $S = R \cup \{ \mathfrak{p} \}$. We also denote $G=\Gal(H/F)$. We fix this notation throughout the paper.

\begin{definition}
For $\sigma \in G$, we define the \textbf{partial zeta function}
\begin{equation}
    \zeta_{R}(\sigma, s)= \sum_{\substack{(\mathfrak{a},R)=1 \\ \sigma_\mathfrak{a}=\sigma}} \N\mathfrak{a}^{-s}.
    \label{pzf}
\end{equation}

Here the sum is over all integral ideals $\mathfrak{a}\subset \mathcal{O}$ that are relatively prime to the elements of $R$ and whose associated Frobenius element $\sigma_\mathfrak{a} \in G$ is equal to $\sigma$.
\end{definition}

Note that the series (\ref{pzf}) converges for $\text{Re}(s)>1$ and has meromorphic continuation to $\mathbb{C}$, regular outside $s=1$. The zeta functions associated to the sets of primes $R$ and $S$
are related by the formula
\[ \zeta_{S}(\sigma, s)= (1-\N\mathfrak{p}^{-s}) \zeta_{R}(\sigma,s). \]
If $K$ is a finite abelian extension of $F$ and $\sigma \in \Gal(K/F)$ we use the notation $\zeta_{R}(K/F, \sigma, s)$ for the partial zeta function defined as above but with the equality $\sigma_\mathfrak{a}=\sigma$ being viewed in $\Gal(K/F)$.

\begin{definition}
Define the group
\begin{equation*}
U_{\mathfrak{p}}=\{ u \in H^\ast : \ \mid u \mid_\mathfrak{P}=1 \ \text{if} \ \mathfrak{P} \ \text{does not divide} \ \mathfrak{p} \}. 
\end{equation*}
\label{updefn}
\end{definition}

Here $\mathfrak{P}$ ranges over all finite and archimedean places of $H$; in particular, each complex conjugation in $H$ acts as an inversion on $U_{\mathfrak{p}}$. We now introduce an auxiliary finite set $T$ of primes of $F$, disjoint from $S$. The partial zeta function associated to the sets $S$ and $T$ is defined by the group ring equation
\begin{equation}
    \sum_{\sigma \in G} \zeta_{S,T}(\sigma, s)[\sigma]= \prod_{\eta \in T}(1-[\sigma_\eta]\N\eta^{1-s}) \sum_{\sigma \in G} \zeta_{S}(\sigma, s)[\sigma].
    \label{pzfT}
\end{equation}

We also assume that the set $T$ contains at least two primes of different residue characteristic or at least one prime $\eta$ with absolute ramification degree at most $l-2$ where $\eta$ lies above $l$. With this in place, the values $\zeta_{S,T}(K/F, \sigma, 0)$ are rational integers for any finite abelian extension $K/F$ unramified outside $S$ and any $\sigma \in \Gal(K/F)$. This was shown by Deligne-Ribet \cite{MR579702} and Cassou-Nogu\'es \cite{MR524276}. The following conjecture was first stated by Tate and called the Brumer-Stark conjecture (Conjecture $5.4$, \cite{MR656067}). We present the formulation given by Gross.

\begin{conjecture}[Conjecture $7.4$, \cite{MR931448}]
Let $\mathfrak{P}$ be a prime in $H$ above $\mathfrak{p}$. There exists an element $u_T \in U_{\mathfrak{p}}$ such that $u_T \equiv 1 \pmod{T}$, and for all $\sigma \in G$, we have 
\begin{equation*}
    \ord_\mathfrak{P}(u_T^\sigma)= \zeta_{R,T}(H/F, \sigma,0).
\end{equation*}
\label{Conj2.5}
\end{conjecture}

Our assumption on $T$ implies that there are no nontrivial roots of unity in $H$ that are congruent to $1$ modulo $T$. Thus, the $\mathfrak{p}$-unit, if it exists, is unique. Note also that our $u_T$ is actually the inverse of the $u$ in (Conjecture $7.4$, \cite{MR931448}).

The conjectural element $u_T \in U_{\mathfrak{p}}$ satisfying Conjecture \ref{Conj2.5} is called the Gross-Stark unit for the data $(S,T,H,\mathfrak{P})$. This conjecture has been recently proved, away from $2$, by Dasgupta-Kakde in \cite{Brumerstark}.

\section{Shintani zeta functions}

Shintani zeta functions are a crucial ingredient in each of the constructions we are studying. The first step in defining these modified zeta functions considers the work of Shintani, initially developed in his paper \cite{MR0427231}, and the definitions of Shintani cones and domains. We establish the necessary notation here.

For each $v \in R_\infty$ we write $\sigma_v : F \rightarrow \mathbb{R}$ and fix the order of these embeddings. We can then embed $F$ into $\mathbb{R}^n$ by $x \mapsto (\sigma_v(x))_{v\in R_\infty}$. We note that $F^\ast$ acts on $\mathbb{R}^n$ with $x \in F$ acting by multiplication by $\sigma_v(x)$ on the $v$-component of any vector in $\mathbb{R}^n$. For linearly independent $v_1, \dots ,v_r \in \mathbb{R}_+^n$, define the simplicial cone 
\[ C(v_1, \dots , v_r)= \left\{ \sum_{i=1}^r c_i v_i \in \mathbb{R}_+^n : c_i>0 \right\}. \]

\begin{definition}
A \textbf{Shintani cone} is a simplicial cone $C(v_1, \dots , v_r)$ generated by elements $v_i \in F \cap \mathbb{R}_+^n$. A \textbf{Shintani set} is a subset of $\mathbb{R}_+^n$ that can be written as a finite disjoint union of Shintani cones. 
\end{definition}

We now give the definition for \textbf{Shintani zeta functions}. Write $\mathfrak{f}$ for the conductor of the extension $H/F$. Let $\mathfrak{b}$ be a fractional ideal of $F$ relatively prime to $S$ and $\overline{T}= \{ q \in \mathbb{Z}, \ \text{prime} : \mathfrak{q} \mid q \ \text{for some} \ \mathfrak{q} \in T \}$. Let $z \in \mathfrak{b}^{-1}$ be such that $z \equiv 1 \pmod{\mathfrak{f}}$, and let ${D}$ be a Shintani set. For each compact open $U \subseteq \mathcal{O}_\mathfrak{p}$, define, for $\text{Re}(s)>1$,
\begin{equation*}
    \zeta_{R}(\mathfrak{b}, {D}, U,s)=\N \mathfrak{b}^{-s} \sum_{\substack{\alpha \in F \cap {D}, \  \alpha \in U \\ (\alpha, R)=1, \alpha \in \mathfrak{b}^{-1} \\ \alpha \equiv 1 \pmod{\mathfrak{f}} }} \N \alpha^{-s} .
    \label{eqn20}
\end{equation*}
We define $\zeta_{R,T}(\mathfrak{b}, {D},U,s)$ in analogy with (\ref{pzfT}) i.e., by the group ring equation
\begin{equation}
    \sum_{\sigma \in G} \zeta_{S,T}(\mathfrak{b}, D, U, s)[\sigma]= \prod_{\eta \in T}(1-[\sigma_\eta]\N\eta^{1-s}) \sum_{\sigma \in G} \zeta_{S}(\mathfrak{b}, D, U, s)[\sigma].
    \label{shintanizetafn}
\end{equation}
It follows from Shintani's work in \cite{MR0427231} that the function $\zeta_{R,T}(\mathfrak{b}, {D},U,s)$ has a meromorphic continuation to $\mathbb{C}$. We now want to define conditions on the set of primes $T$ and the Shintani set $D$ to allow our Shintani zeta functions to be integral at $0$.

\begin{definition}
A prime ideal $\eta$ of $F$ is called \textbf{good} for a Shintani cone $C$ if
\begin{itemize}
    \item $ \N \eta$ is a rational prime $l$; and 
    \item the cone $C$ may be written $C=C(v_1, \dots , v_r)$ with $v_i \in \mathcal{O}$ and $v_i \nin \eta$.
\end{itemize}
We also say that $\eta$ is \textbf{good} for a Shintani set $D$ if $D$ can be written as a finite disjoint union of Shintani cones for which $\eta$ is good.
\end{definition}



\begin{definition}
The set $T$ is \textbf{good} for a Shintani set ${D}$ if ${D}$ can be written as a finite disjoint union of Shintani cones $D=\cup C_i$ so that for each cone $C_i$, there are at least two primes in $T$ that are good for $C_i$ (necessarily of different residue characteristic by our earlier assumption) or one prime $\eta \in T$ that is good for $C_i$ such that $\N\eta \geq n+2$.
\end{definition}

\begin{remark}
Given any Shintani set $D$, it is possible to choose a set of primes $T$ such that $T$ is good for $D$. In fact, all but a finite number of prime ideals will be good for a given Shintani set.
\end{remark}

We can now note the required property to allow our Shintani zeta functions to be integral at zero. This follows from the following proposition of Dasgupta.

\begin{proposition}[Proposition 3.12, \cite{MR2420508}]
If the set of primes $T$ contains a prime $\eta$ that is good for a Shintani cone $C$ and $\N \eta =l$, then
\[ \zeta_{R,T}(\mathfrak{b}, C,U,0) \in \mathbb{Z}[l/l]. \]
Furthermore, the denominator of $\zeta_{R,T}(\mathfrak{b}, C,U,0)$ is at most $l^{n/(l-1)}$.
\label{prop3.12}
\end{proposition}

As is noted by Dasgupta at the top of p.15 in \cite{MR2420508}, the corollary below follows easily from Proposition \ref{prop3.12}.

\begin{corollary}
If the set of primes $T$ is good for a Shintani set $D$, then
\[ \zeta_{R,T}(\mathfrak{b}, D,U,0) \in \mathbb{Z}. \]
\end{corollary}

We define a $\mathbb{Z}$-valued measure $\nu(\mathfrak{b}, {D})$ on $\mathcal{O}_\mathfrak{p}$ by
\begin{equation}
    \nu(\mathfrak{b}, {D},U) \coloneqq \zeta_{R,T}(\mathfrak{b}, {D},U,0)
    \label{eqn21}
\end{equation}
for $U \subseteq \mathcal{O}_\mathfrak{p}$ compact open.

We are mostly interested in a particular type of Shintani set, one which is a fundamental domain for the action of $E_+(\mathfrak{f})$.

\begin{definition}
We call a Shintani set ${D}$ a \textbf{Shintani domain} if ${D}$ is a fundamental domain for the action of $E_+(\mathfrak{f})$
on $\mathbb{R}_+^n$. That is, when
\[ \mathbb{R}_+^n=\bigcup_{\epsilon \in E_+(\mathfrak{f})} \epsilon {D} \quad \text{(disjoint union).} \]
\end{definition}

The existence of such domains follows the work of Shintani, in particular from (Proposition $4$, \cite{MR0427231}). We note here some simple equalities which follow from the definitions, more details are given in \S 3.3 of \cite{MR2420508}. We write $G_\mathfrak{f}$ for the narrow ray class group of conductor $\mathfrak{f}$. Let $e$ be the order of $\mathfrak{p}$ in $G_\mathfrak{f}$, and suppose that $\mathfrak{p}^e = (\pi)$ with $\pi \equiv 1 \pmod{\mathfrak{f}}$ and $\pi$ totally positive. Let $\mathcal{D}$ be a Shintani domain and write $\mathbb{O}= \mathcal{O}_\mathfrak{p}-\pi \mathcal{O}_\mathfrak{p} $. Then,
\[
    \nu(\mathfrak{b}, \mathcal{D},\mathbb{O}) = \zeta_{S,T}(H/F, \mathfrak{b}, 0) = 0, \quad \text{and} \quad
    \nu(\mathfrak{b}, \mathcal{D},\mathcal{O}_\mathfrak{p}) = \zeta_{R,T}(H/F, \mathfrak{b}, 0).
\]

We now give two technical definitions which are necessary in the definition of Dasgupta's explicit formula and recall a useful lemma which is used repeatedly in the proof of our later results. We will also generalise to working with $V \subset E_+(\mathfrak{f})$ rather than with just $E_+(\mathfrak{f})$.

\begin{definition}
Let $V \subset E_+(\mathfrak{f})$ be a finite index free subgroup of rank $n-1$. We call a Shintani set ${D}$ a \textbf{Colmez domain} for $V$ if ${D}$ is a fundamental domain for the action of $V$
on $\mathbb{R}_+^n$. That is, when
\[ \mathbb{R}_+^n=\bigcup_{\epsilon \in V} \epsilon {D} \quad \text{(disjoint union).} \]
\end{definition}

We note that in the definition of a Colmez domain we allow ourselves to work with $V= E_+(\mathfrak{f})$, thus the definition includes Shintani domains.

\begin{proposition}
Let $V \subset E_+(\mathfrak{f})$ be a finite index free subgroup of rank $n-1$. Let ${D}$ and ${D}^\prime$ be Colmez domains for $V$. We may write ${D}$ and ${D}^\prime$ as finite disjoint unions of the same number of simplicial cones 
\begin{equation}
    {D}= \bigcup_{i=1}^d C_i, \quad {D}^\prime = \bigcup_{i=1}^d C_i^\prime,
    \label{eqn9}
\end{equation}
with $C_i^\prime = \epsilon_i C_i$ for some $\epsilon_i \in V$, $i=1, \dots , d$.
\end{proposition}

\begin{proof}
Proposition 3.15 of \cite{MR2420508} proves this result when $V=E_+(\mathfrak{f})$. The proof of this proposition is analogous.
\end{proof}

A decomposition as in (\ref{eqn9}) is called a \textbf{simultaneous decomposition} of the Colmez domains $({D}, {D}^\prime)$. 

\begin{definition}
Let $(D,D^\prime)$ be a pair of Colmez domains. A set $T$ is \textbf{good} for the pair $({D}, {D}^\prime)$ if there is a simultaneous decomposition as in (\ref{eqn9}) such that for each cone $C_i$, there are at least two primes in $T$ that are good for $C_i$, or there is one prime $\eta \in T$ that is good for $C_i$ such that $\N \eta \geq n+2$.
\end{definition}

\begin{definition}
Let $D$ be a Colmez domain. If $\beta \in F^\ast$ is totally positive, then $T$ is $\beta$-\textbf{good} for ${D}$ if $T$ is good for the pair $({D}, \beta^{-1} {D})$.
\end{definition}

The following lemma is used throughout the remainder of this paper.
\begin{lemma}[Lemma 3.20, \cite{MR2420508}]
 Let $D$ be a Shintani set and $U$ a compact open subset of $\mathcal{O}_\mathfrak{p}$. Let $\mathfrak{b}$ be a fractional ideal of $F$, and let $\beta \in F^\ast$ be totally positive so that $\beta \equiv 1 \pmod{\mathfrak{f}}$ and $\ord_\mathfrak{p}(\beta) \geq 0$. Suppose that $\mathfrak{b}$ and $\beta$ are relatively prime to $R$ and that $\mathfrak{b}$ is also relatively prime to  $\overline{T}$. Let $\mathfrak{q}= (\beta)\mathfrak{p}^{-\ord_\mathfrak{p}(\beta)}$. Then
 \[ \zeta_{R,T}(\mathfrak{bq}, {D},U,0) = \zeta_{R,T}(\mathfrak{b}, \beta {D},\beta U,0). \]
 \label{changeofvariable}
\end{lemma}

We end this section with a Lemma of Colmez which allows us to give an explicit Colmez domain. Let $\alpha$ be, up to a sign, one of the standard basis vectors of $\mathbb{R}^n$ then we note that its ray ($\alpha\mathbb{R}_+$) is preserved by the action of $\mathbb{R}_+^n$. We define $\overline{C}_\alpha(v_1, \dots , v_r)$ to be the union of the cone $C(v_1, \dots , v_r)$ with the boundary cones that are brought into the interior of the cone by a small perturbation by $\alpha$, i.e., the set whose characteristic function is given by 
\begin{equation}
    \mathbbm{1}_{\overline{C}_\alpha(v_1, \dots , v_r)}(x) = \lim_{h \rightarrow 0^+} \mathbbm{1}_{C(v_1, \dots, v_r)}(x+h\alpha). 
    \label{overlinec}
\end{equation}
Throughout this paper we will use the notation
\[ [x_1 \mid \dots \mid x_{n-1}] = (1, x_1, x_1 x_2, \dots , x_1 \dots x_{n-1}). \]
Let $x_1, \dots , x_{n-1} \in E_+(\mathfrak{f})$. We define the sign map $\delta : E_+(\mathfrak{f})^{n} \rightarrow \{ -1,0,1 \} $ such that, 
\begin{equation}
    \delta(x_1, \dots , x_n)=  \sign(\det  (\omega(x_1, \dots , x_n))),
    \label{epsilonnotation}
\end{equation}
where $\omega(x_1, \dots , x_n)$ denotes the $n \times n$ matrix whose columns are the images of the $x_i$ in $\mathbb{R}_+^n$. Note that we have the convention that $\sign(0)=0$.

\begin{lemma}[Lemma 2.2, \cite{MR922806}]
Let $\alpha$ be, up to a sign, one of the standard basis vectors of $\mathbb{R}^n$. Let $\varepsilon_1, \dots, \varepsilon_{n-1} \in E_+(\mathfrak{f})$ such that $V=\langle \varepsilon_1, \dots, \varepsilon_{n-1}  \rangle \subset E_+(\mathfrak{f}) $ is a free subgroup of rank $n-1$ and finite index. Suppose that for all $\tau \in S_{n-1}$ we have
\[ \delta([\varepsilon_{\tau(1)} \mid \dots \mid \varepsilon_{\tau(n-1)}])=\sign(\tau). \]
Then the Shintani set
\[ {D}= \bigcup_{\tau \in S_{n-1}}\overline{C}_{\alpha}([\varepsilon_{\tau(1)} \mid \dots \mid \varepsilon_{\tau(n-1)}]),  \]
is a Colmez domain for $V$.
\label{colmezlemma}
\end{lemma}

\section{The Gross regulator matrix}

We continue to let $F$ denote a totally real field of degree $n$, and let 
\[ \chi : \Gal(\overline{F}/ F) \rightarrow \overline{ \mathbb{Q}} \]
be a totally odd character. We fix embeddings $\overline{\mathbb{Q}} \subset \mathbb{C}$ and $\overline{\mathbb{Q}} \subset \mathbb{C}_p $, so $\chi$ may be viewed as taking values in $\mathbb{C}$ or $\mathbb{C}_p$. In this section, we let $H$ denote the fixed field of the kernel of $\chi $. Note that $H$ is a finite cyclic CM extension of $F$. As usual, we view $\chi$ also as a multiplicative map on the semigroup of integral fractional ideals of $F$ by defining $\chi(\mathfrak{q}) = \chi (\sigma_\mathfrak{q})$ if $\mathfrak{q}$ is unramified in $H$ and $\chi(\mathfrak{q})=0$ if $\mathfrak{q}$ is ramified in $H$. Let $S_p$ denote the set of places of $F$ lying above $p$ that split completely in $H$. For each prime $\mathfrak{p} \in S_p$, we define the group of $\mathfrak{p}$-units, $U_\mathfrak{p}$, as in Definition \ref{updefn}. We then write
\begin{align*}
    U_{\mathfrak{p}, \chi} &\coloneqq (U_\mathfrak{p} \otimes \overline{\mathbb{Q}})^{\chi^{-1}} \\
    &= \{ u \in U_\mathfrak{p} \otimes \overline{\mathbb{Q}} \mid \sigma(u) = u \otimes \chi^{-1}(\sigma) \ \text{for all} \ \sigma \in \Gal(H/F) \}.
\end{align*}
The Galois equivarient form of Dirichlet's unit theorem implies that
\[ \text{dim}_{\overline{\mathbb{Q}}} U_{\mathfrak{p}, \chi} = \begin{cases} 1 & \text{if} \ \mathfrak{p} \in S_p, \\ 0 & \text{otherwise}. \end{cases} \]
Let $u_{\mathfrak{p}, \chi}$ denote any generator (i.e., non-zero element) of $U_{\mathfrak{p},\chi}$. Consider the continuous homomorphisms
\begin{align}
    o_\mathfrak{p} \coloneqq \ord_\mathfrak{p} &: F_\mathfrak{p}^\ast \rightarrow \mathbb{Z} \label{ordmap} \\
    l_\mathfrak{p} \coloneqq \log_p \circ \text{Norm}_{F_\mathfrak{p}/\mathbb{Q}_p} &: F_\mathfrak{p}^\ast \rightarrow \mathbb{Z}_p . \label{logmap}
\end{align}
Suppose we choose for each $\mathfrak{p} \in S_p$, a prime $\mathfrak{P}_\mathfrak{p}$ of $H$ lying above $\mathfrak{p}$. Then, for $\mathfrak{p}, \mathfrak{q} \in S_p$, via
\[ U_\mathfrak{p} \subset H \subset H_{\mathfrak{P}_\mathfrak{q}} \cong F_\mathfrak{q}, \]
we can evaluate $o_\mathfrak{q}$ and $l_\mathfrak{q}$ on elements of $U_\mathfrak{p}$, and extend by linearity to maps
\[ o_\mathfrak{q}, l_\mathfrak{q}:U_{\mathfrak{p}, \chi} \rightarrow \mathbb{C}_p. \]
Define the ratio
\[ \mathcal{L}_\text{alg}(\chi)_{\mathfrak{p}, \mathfrak{q}} = -\frac{l_\mathfrak{q}(u_{\mathfrak{p}, \chi})}{o_\mathfrak{q}(u_{\mathfrak{p}, \chi})}, \]
which is clearly independent of the choice of $u_{\mathfrak{p}, \chi} \in U_{\mathfrak{p}, \chi}$. \textbf{Gross's regulator}, $\mathcal{R}_p(\chi)$, is the determinant of the $\# S_p \times \# S_p $ matrix whose entries are given by these values:
\[ \mathcal{R}_p(\chi) \coloneqq \det (\mathcal{M}_p(\chi)), \quad \text{where} \quad \mathcal{M}_p(\chi) \coloneqq (\mathcal{L}_{\text{alg}}(\chi)_{\mathfrak{p}, \mathfrak{q}})_{\mathfrak{p}, \mathfrak{q} \in S_p}. \]
We refer to $\mathcal{M}_p(\chi)$ as the \textbf{Gross regulator matrix}. More generally, for any subset $J \subset S_p$, the principle minor of $\mathcal{M}_p(\chi)$ corresponding to $J$ is defined by
\[ \mathcal{R}_p(\chi)_J \coloneqq \det (\mathcal{L}_\text{alg}(\chi)_{\mathfrak{p}, \mathfrak{q}})_{\mathfrak{p}, \mathfrak{q} \in J}. \]
We note that both $\mathcal{R}_p(\chi)$ and $\mathcal{R}_p(\chi)_J$ are independent of all choices. This is explained in more detail in \S 1 of \cite{MR3968788}. In \S 3 of \cite{MR3968788}, Dasgupta-Spie\ss \ constructed a conjectural formula for the value  $\mathcal{R}_p(\chi)_J$ via group cohomology (Conjecture 3.1, \cite{MR3968788}). If we take $J=\{ \mathfrak{p} \}$ for some $\mathfrak{p} \in S_p$ then the value of $\mathcal{R}_p(\chi)_\mathfrak{p}$ is the diagonal entry at $\mathfrak{p}$ of the Gross regulator matrix i.e.,
\[ \mathcal{R}_p(\chi)_\mathfrak{p} = \mathcal{L}_\text{alg}(\chi)_{\mathfrak{p}, \mathfrak{p}} = -\frac{l_\mathfrak{p}(u_{\mathfrak{p}, \chi})}{o_\mathfrak{p}(u_{\mathfrak{p}, \chi})}. \]
Since we are only concerned with their conjecture in this case we henceforth fix a choice of $\mathfrak{p} \in S_p$. The remainder of this section is leading us to define their formula in this case.

\subsection{The Eisenstein cocycle}

We now define the Eisenstein cocycle. Let $k$ denote the cyclotomic field generated by the values of $\chi$. Now let $\mathfrak{P}$ be the prime of $k$ above $p$ corresponding to the embeddings $k \subset \overline{\mathbb{Q}} \subset \mathbb{C}_p$, where the second embedding is the one fixed at the start of this section. Let $K=k_\mathfrak{P}$, and write $\mathcal{O}_K$ for it's ring of integers. As before write $\mathfrak{f}$ for the conductor of the extension $H/F$. Let $\lambda$ be a prime of $F$ such that $\N \lambda =l$ for a prime number $l\in \mathbb{Z}$ and $l \geq n+2$. We assume that no primes in $S$ have residue characteristic equal to $l$. Let $E_+(\mathfrak{f})_\mathfrak{p}$ denote the group of $\mathfrak{p}$-units of $F$ which are congruent to $1 \pmod{\mathfrak{f}}$. We note that $E_+(\mathfrak{f})_\mathfrak{p}$ is free of rank $n$. For $x_1, \dots , x_n \in E_+(\mathfrak{f})_{\mathfrak{p}}$, a fractional ideal $\mathfrak{b}$ coprime to $S$ and $l$, and compact open $U \subset F_\mathfrak{p}$ we put
\[ \nu_{\mathfrak{b}, \lambda}^\mathfrak{p}(x_1, \dots , x_n)(U) \coloneqq \delta(x_1, \dots , x_n) \zeta_{R, \lambda}(\mathfrak{b}, \overline{C}_{e_1}(x_1, \dots, x_n), U,0). \]
In the above, $e_1$ denotes the basis vector $(1,0, \dots ,0) \in \mathbb{R}^n$ and $\delta$ is defined as, in (\ref{epsilonnotation}). We recall the definition of the Shintani zeta function from (\ref{shintanizetafn}) and the Shintani set $\overline{C}_{e_1}(x_1, \dots, x_n)$ from (\ref{overlinec}). Then, $\nu_{\mathfrak{b}, \lambda}^\mathfrak{p}$ is a homogeneous $(n-1)$-cocycle on $E_+(\mathfrak{f})_{\mathfrak{p}}$ with values in the space of $\mathbb{Z}$-distribution on $F_\mathfrak{p}$. This follows from Theorem $2.6$ of \cite{MR3351752}. Hence, we have defined a class
\[ \omega_{\mathfrak{f}, \mathfrak{b}, \lambda}^\mathfrak{p} \coloneqq [ \nu_{\mathfrak{b}, \lambda}^\mathfrak{p} ] \in H^{n-1}(E_+(\mathfrak{f})_{\mathfrak{p}} , \text{Meas}(F_\mathfrak{p},K)), \]
where we define
\[ \text{Meas}(F_\mathfrak{p},K) \coloneqq \Hom ( C_c(F_\mathfrak{p}, \mathbb{Z}) , \mathcal{O}_K ) \otimes_{\mathcal{O}_K} K, \]
having let $C_c(F_\mathfrak{p}, \mathbb{Z})$ denote the set of compactly supported continuous functions from $F_\mathfrak{p}$ to $\mathbb{Z}$. We also consider
\[ \omega_{\chi,\lambda}^\mathfrak{p} = \sum_{[\mathfrak{b}] \in G_\mathfrak{f}/ \langle \mathfrak{p} \rangle} \frac{\chi(\mathfrak{b})}{1-\chi(\lambda)l} \omega_{\mathfrak{f}, \mathfrak{b}, \lambda}^\mathfrak{p} \in H^{n-1}(E_+(\mathfrak{f})_{\mathfrak{p}} , \text{Meas}(F_\mathfrak{p},K)), \]
where the sum ranges over a system of representatives of $G_\mathfrak{f}/ \langle \mathfrak{p} \rangle$. For more details on this construction, see \S 3.3 of \cite{MR3968788}.

\subsection{1-cocycles attached to homomorphisms}

Let $g:F_\mathfrak{p}^\ast \rightarrow K$ be a continuous homomorphism. We want to define a cohomology class $c_g \in H^1(F_\mathfrak{p}^\ast , C_c (F_\mathfrak{p}, K))$ attached to $g$. We define an $F_\mathfrak{p}^\ast$-action on $C_c(F_\mathfrak{p}^\ast, \mathbb{Z})$ by $(xf)(y)=f(x^{-1}y)$. Here $C_c(F_\mathfrak{p}^\ast, \mathbb{Z})$ is the space of compactly supported continuous functions from $F_\mathfrak{p}^\ast$ to $\mathbb{Z}$. The following definition is due to Spie\ss \ and first appears in Lemma $2.11$ of \cite{MR3179573}. This definition is crucial in making the construction of Dasgupta-Spie\ss's cohomological formula work and we also remark that the definition is unusual in that it appears as though the cocycle $z_g$ should be a coboundary. However, it may not be a coboundary since $g$ does not necessarily extend to a continuous function on $F_\mathfrak{p}$.
 
\begin{definition}
Let $g:F_\mathfrak{p}^\ast \rightarrow K$ be a continuous homomorphism as above and let $f \in C_c(F_\mathfrak{p},\mathbb{Z})$ such that $f(0)=1$. We define $c_g$ to be the class of the cocycle $z_{f,g}:F_\mathfrak{p}^\ast \rightarrow C^\diamond (F_\mathfrak{p},A)$ where $z_{f,g} (x)=``(1-x)(g \cdot f)"$, or more precisely
\begin{equation}
    z_{f,g} (x)(y)=(xf)(y) \cdot g(x) + ((f-xf)\cdot g)(y)
    \label{zeqn}
\end{equation}
for $x \in F_\mathfrak{p}^\ast$ and $y \in F_\mathfrak{p}$.
\label{1cocycle}
\end{definition}

The second term in (\ref{zeqn}) is allowed to be evaluated at $0 \in F_\mathfrak{p}$ since we can extend continuously the function from $F_\mathfrak{p}^\ast$ to ${F}_\mathfrak{p}$ as \[(f-xf)(0)=0.\]

Using this definition, we are able to define an element $c_g \coloneqq [z_{f,g}] \in H^1(F_\mathfrak{p}^\ast, C_c(F_\mathfrak{p},K))$ for any continuous homomorphism $g:F_\mathfrak{p}^\ast \rightarrow K$ and any $f \in C_c( F_\mathfrak{p}, \mathbb{Z})$ with $f(0)=1$. We note that the class is independent of the choice of $f \in C_c( F_\mathfrak{p}, \mathbb{Z})$ with $f(0)=1$. In particular, we can consider the classes $c_{o_\mathfrak{p}},c_{l_\mathfrak{p}} \in H^1(F_\mathfrak{p}^\ast, C_c(F_\mathfrak{p},K))$. The homomorpisms $o_\mathfrak{p}$ and $l_\mathfrak{p}$ are as defined in (\ref{ordmap}) and (\ref{logmap}).

For the results we want to show, Definition \ref{1cocycle} is all that we require. For more information on these objects, see \S 3.2 of \cite{MR3861805} and \S 3.1 of \cite{MR3968788}.

\subsection{The diagonal entries}

We now give the definition of Dasgupta-Spie\ss's conjectural formula for the diagonal entries of the Gross regulator matrix. Recall that we have defined the following objects:
\[ c_{o_\mathfrak{p}},c_{l_\mathfrak{p}} \in H^1(F_\mathfrak{p}^\ast, C_c(F_\mathfrak{p},K)) \quad \text{and} \quad \omega_{\chi,\lambda}^\mathfrak{p}\in H^{n-1}(E_+(\mathfrak{f})_{\mathfrak{p}} , \text{Meas}(F_\mathfrak{p},K)).  \]
We now consider $H_{n}(E_+(\mathfrak{f})_{\mathfrak{p}}, \mathbb{Z})$. By Dirichlet's unit theorem, $E_+(\mathfrak{f})_{\mathfrak{p}}$ is free abelian of rank $n$. Hence,  $H_{n}(E_+(\mathfrak{f})_{\mathfrak{p}}, \mathbb{Z}) \cong \mathbb{Z}$. We are thus able to choose a generator $\vartheta^\prime \in H_{n}(E_+(\mathfrak{f})_{\mathfrak{p}}, \mathbb{Z})$. Cap and cup products are a crucial element of Dasgupta-Spie\ss's formula. For the definitions of these products, refer to chapter 6 of \cite{MR1324339}.

\begin{definition}[Proposition 3.6, \cite{MR3968788}]
Let $\vartheta^\prime \in H_{n}(E_+(\mathfrak{f})_{\mathfrak{p}}, \mathbb{Z})$ be a generator. Then, we define
\begin{equation}
    \mathcal{R}_p(\chi)_{\mathfrak{p},\text{an}} \coloneqq (-1)\frac{c_{l_\mathfrak{p}} \cap (\omega_{\chi,\lambda}^\mathfrak{p} \cap \vartheta^\prime)}{c_{o_\mathfrak{p}} \cap (\omega_{\chi,\lambda}^\mathfrak{p} \cap \vartheta^\prime)}.
\end{equation}
\label{grossanreg2}
\end{definition}

The ``$an$'' notation here is only used to distinguish the formula $\mathcal{R}_p(\chi)_{\mathfrak{p},\text{an}}$ from the algebraic quantity $\mathcal{R}_p(\chi)_{\mathfrak{p}}$. We note that in \cite{MR3968788}, a different formula is initially given for the quantity $\mathcal{R}_p(\chi)_{\mathfrak{p},\text{an}}$. In (Proposition 3.6, \cite{MR3968788}) it is shown that the initial expression is equal to the quantity we define in Definition \ref{grossanreg2}. Since the formula we gave in Definition \ref{grossanreg2} is more useful for our calculations we shall give it here as the definition of $\mathcal{R}_p(\chi)_{\mathfrak{p},\text{an}}$. In \cite{MR3968788}, Dasgupta-Spie\ss \ conjectured that their formula $\mathcal{R}_p(\chi)_{\mathfrak{p},\text{an}}$ is in fact equal to $\mathcal{R}_p(\chi)_{\mathfrak{p}}$.

\begin{conjecture}[Conjecture 3.1, \cite{MR3968788}]
For each $\mathfrak{p} \in S_p$, we have $\mathcal{R}_p(\chi)_\mathfrak{p}= \mathcal{R}_p(\chi)_{\mathfrak{p}, \text{an}}$.
\label{conjDS}
\end{conjecture}

It is worth noting that in \cite{MR3968788} Dasgupta-Spie\ss \ conjectured a more general formula than the version we write above. This formula is conjectured to be equal to the value $\mathcal{R}_p(\chi)_J$ for any subset $J \subseteq S_p$. Since we only work in the case $J=\{ \mathfrak{p} \}$ we have only given the definition in this case. We remark also that if we are in rank 1, i.e., $\mid S_p \mid = 1$ then Conjecture \ref{conjDS} matches exactly with Conjecture 3.1 in \cite{MR3968788}.

In the next section, we study an analytic formula, conjectured by Dasgupta in \cite{MR2420508}, for the image of the Gross-Stark units in $F_\mathfrak{p}$. This allows us to give a formula for the image of $u_{\mathfrak{p}, \chi}$ in $F_\mathfrak{p}^\ast \otimes K$ and hence will give us another formula for the quantity $\mathcal{R}_p(\chi)_\mathfrak{p}$. The main result of this paper is that these two formulas for $\mathcal{R}_p(\chi)_\mathfrak{p}$ are equal.

\section{The multiplicative integral formula}

\begin{definition}
Let $I$ be an abelian topological group that may be written as an inverse limit of discrete groups
\[ I= \varprojlim I_\alpha . \]
Denote the group operation on $I$ multiplicatively. For each $i \in I_\alpha$, denote by $U_i$ the open subset of $I$ consisting of the elements that map to $i$ in $I_\alpha$. Suppose that $G$ is a compact open subset of a quotient of $\mathbb{A}_F^\ast$  . Let $f:G \rightarrow I$ be a continuous map, and let $\mu$ be a $\mathbb{Z}$-valued measure. We define the \textbf{multiplicative integral}, written with a cross through the integration sign, by 
\begin{equation*}
    \multint_G f(x) d \mu (x)= \varprojlim \prod_{i \in I_\alpha}i^{\mu (f^{-1}(U_i))} \in I.
\end{equation*}
\end{definition}

The first definition we make towards the formula is that of an element of $E_+(\mathfrak{f})$. We refer to this construction as the error term. After the definition, we check that it is well defined.

\begin{definition}
Let $\mathcal{D}$ be a Shintani domain, and assume that $T$ is $\pi$-good for $\mathcal{D}$. Define the \textbf{error term} 
\begin{equation}
    \epsilon(\mathfrak{b}, \mathcal{D}, \pi ) \coloneqq \prod_{\epsilon \in E_+(\mathfrak{f})} \epsilon^{\nu(\mathfrak{b}, \epsilon \mathcal{D} \cap \pi^{-1} \mathcal{D}, \mathcal{O}_\mathfrak{p})}.
    \label{errorterm}
\end{equation}
\end{definition}

By Lemma $3.14$ of \cite{MR2420508}, only finitely many of the exponents in (\ref{errorterm}) are nonzero. Proposition $3.12$ of \cite{MR2420508} and the assumption that $T$ is $\pi$-good for $\mathcal{D}$ imply that the exponents are integers. We recall the definition of the measure $\nu$ from (\ref{eqn21}). We are now ready to write down Dasgupta's conjectural formula. We note that for any Shintani domain $\mathcal{D}$ we can always impose that $T$ is $\pi$-good for $\mathcal{D}$ by adding a finite number of primes to $T$. Henceforth, we shall assume that we are in this case. We now give the main definition of this section.
 
\begin{definition}
Let $\mathcal{D}$ be a Shintani domain, and assume that $T$ is $\pi$-good for $\mathcal{D}$. Define 
\begin{equation*}
    u_{\mathfrak{p},T}(\mathfrak{b}, \mathcal{D}) \coloneqq \epsilon(\mathfrak{b}, \mathcal{D}, \pi ) \pi^{\zeta_{R,T}(H/F, \mathfrak{b},0)} \multint_\mathbb{O} x d \nu (\mathfrak{b}, \mathcal{D},x) \in F^\ast_\mathfrak{p}.
\end{equation*}
\label{defnformula}
\end{definition}

As our notation suggests, we have the following proposition.

\begin{proposition}[Proposition 3.19, \cite{MR2420508}]
The element $u_{\mathfrak{p},T}(\mathfrak{b}, \mathcal{D})$ does not depend on the choice of generator $\pi$ of $\mathfrak{p}^e$.
\label{prop3.19}
\end{proposition}

Dasgupta made the following conjecture concerning his construction.

\begin{conjecture}[Conjecture 3.21, \cite{MR2420508}]
Let $e$ be the order of $\mathfrak{p}$ in $G_\mathfrak{f}$, and suppose that $\mathfrak{p}^e=(\pi)$ with $\pi$ totally positive and $\pi \equiv 1 \pmod{\mathfrak{f}}$. Let $\mathcal{D}$ be a Shintani domain, and let $T$ be $\pi$-good for $\mathcal{D}$. Let $\mathfrak{b}$ be a fractional ideal of $F$ relatively prime to $S$ and $\overline{T}$. We have the following.
\begin{enumerate}
    \item The element $u_{\mathfrak{p},T}(\mathfrak{b}, \mathcal{D}) \in F_\mathfrak{p}^\ast$ depends only on the class of $\mathfrak{b} \in G_\mathfrak{f}/ \langle \mathfrak{p} \rangle$ and no other choices, including the choice of $\mathcal{D}$, and hence may be denoted $u_{\mathfrak{p},T}(\sigma_\mathfrak{b})$, where $\sigma_\mathfrak{b} \in \Gal(H/F)$.
    
    \item The element $u_{\mathfrak{p},T}(\sigma_\mathfrak{b})$ lies in $U_\mathfrak{p}$, and $u_{\mathfrak{p},T}(\sigma_\mathfrak{b}) \equiv 1 \pmod{T}$. 
    
    \item Shimura reciprocity law: For any fractional ideal $\mathfrak{a}$ of $F$ prime to $S$ and to $\char$ $T$, we have 
    \begin{equation*}
        u_{\mathfrak{p},T}(\sigma_\mathfrak{ab})=u_{\mathfrak{p},T}(\sigma_\mathfrak{b})^{\sigma_\mathfrak{a}}.
    \end{equation*}
\end{enumerate}
\label{conj3.21}
\end{conjecture}

Recent work of Dasgupta-Kakde (Theorem 1.6, \cite{intgrossstark}) proved the above conjecture up to a root of unity under the assumption:
\begin{equation}
    p \ \text{is odd and} \ H \cap F(\mu_{p^\infty}) \subset H^+, \ \text{the maximal totally real subfield of} \ H.
    \label{assumptionforDK}
\end{equation}
The main result of their paper is the $p$-part of the integral Gross-Stark conjecture (Theorem 1.4 \cite{intgrossstark}). The Gross-Stark Conjecture first appears in (Conjecture $7.6$, \cite{MR931448}). Conjecture \ref{conj3.21}, up to a root of unity, then follows from Theorem 5.18 of \cite{MR2420508}.

\section{Comparing the formulas}

Let $\chi$ and $H$ be as given at the start of \S 4, and $\lambda$ as given at the start of \S 4.1 . Let $\mathfrak{p} \in S_p$. In \cite{MR3968788}, Dasgupta-Spie\ss \ conjectured the following.

\begin{conjecture}[Remark 4.5, \cite{MR3968788}]
Conjecture \ref{conjDS} is consistent with Conjecture \ref{conj3.21}, i.e., we have
\begin{equation}
    \mathcal{R}_p(\chi)_{\mathfrak{p},\text{an}} =  -\frac{l_\mathfrak{p}(\mathcal{U}_{\mathfrak{p},\chi})}{o_\mathfrak{p}(\mathcal{U}_{\mathfrak{p},\chi})},
    \label{lasteqn}
\end{equation}
where we define
\begin{equation}
    \mathcal{U}_{\mathfrak{p},\chi} = \sum_{[\mathfrak{b}]\in G_\mathfrak{f}/ \langle \mathfrak{p} \rangle} u_{\mathfrak{p},\lambda}(\mathfrak{b},\mathcal{D}) \otimes \chi(\mathfrak{b})/(1- \chi(\lambda)l).
    \label{Ueqn}
\end{equation}
Here, the sum ranges over a set of representatives $\mathfrak{b}$ for $G_\mathfrak{f}/ \langle \mathfrak{p} \rangle$ with $\mathfrak{b}$ relatively prime to $\mathfrak{f}$, $R$ and $l$, and $\mathcal{D}$ is a Shintani domain.
\label{thmforalln}
\end{conjecture}

\begin{remark}
It follows from (Theorem 1.6, \cite{intgrossstark}) that the right hand side of (\ref{lasteqn}) is independent of the choices for $\mathfrak{b}$ and $\mathcal{D}$. The independence of $\lambda$ in the sum in (\ref{Ueqn}) follows from Lemma 5.4 of \cite{MR2420508}.
\end{remark}

In \cite{MR3968788}, Dasgupta-Spie\ss \ proved Conjecture \ref{thmforalln} in the case $n=2$ (recall that $n$ is the degree of our field $F$). The main result of this paper is the proof of Conjecture \ref{thmforalln} in the case $n=3$. Note that in this case we have that $E_+(\mathfrak{f})$ is free of rank $2$. We will show the following theorem.

\begin{theorem}
Let $F$ be a totally real field of degree $3$. Then, Conjecture \ref{conjDS} is consistent with Conjecture \ref{conj3.21}, i.e., we have
\[ \mathcal{R}_p(\chi)_{\mathfrak{p},\text{an}} =  -\frac{l_\mathfrak{p}(\mathcal{U}_{\mathfrak{p},\chi})}{o_\mathfrak{p}(\mathcal{U}_{\mathfrak{p},\chi})} \]
where $\mathcal{U}_{\mathfrak{p},\chi}$ is as defined in Conjecture \ref{thmforalln} and $\mathcal{D}$ is a Shintani domain.
\label{thmforneq3}
\end{theorem}

It is worth noting that in the $n=2$ case the proof of the result is much shorter due to the simple nature of the Shintani domains when $F$ is of degree $2$. The main challenge we have to overcome is working with Shintani domains which live in $\mathbb{R}^3_+$ rather than in $\mathbb{R}_+^2$. This difficulty is expanded on further in the later sections.

Theorem \ref{thmforneq3} combined with (Theorem 1.6, \cite{intgrossstark}) of Dasgupta-Kakde allows us to prove Conjecture \ref{conjDS} when $F$ is of degree $3$. Thus, we have the following corollary. We note that the corollary below closes some cases of (Conjecture 3.1, \cite{MR3968788}). In particular, for the $1 \times 1$ principle minors when $F$ is a cubic field.

\begin{corollary}
Let $F$ be a totally real field of degree $3$ and suppose that condition (\ref{assumptionforDK}) holds. Then, for each $\mathfrak{p} \in S_p$, we have $\mathcal{R}_p(\chi)_\mathfrak{p}= \mathcal{R}_p(\chi)_{\mathfrak{p}, \text{an}}$.
\end{corollary}

\begin{proof}
We apply Theorem 1.6 in \cite{intgrossstark} to Theorem \ref{thmforneq3}.
\end{proof}

We now observe that, due to Definition \ref{grossanreg2}, Conjecture \ref{thmforalln} follows immediately from the following conjecture.
\begin{conjecture}
Let $[\mathfrak{b}] \in G_\mathfrak{f}$ and $\mathcal{D}$ a Shintani domain. Then, 
\[ l_\mathfrak{p}(u_{\mathfrak{p}, \lambda}(\mathfrak{b},  \mathcal{D}))= \pm c_{l_\mathfrak{p}} \cap (\omega_{\mathfrak{f}, \mathfrak{b}, \lambda}^\mathfrak{p} \cap \vartheta^\prime), \quad o_\mathfrak{p}(u_{\mathfrak{p}, \lambda}(\mathfrak{b},  \mathcal{D})) = \pm c_{o_\mathfrak{p}} \cap (\omega_{\mathfrak{f}, \mathfrak{b}, \lambda}^\mathfrak{p} \cap \vartheta^\prime). \]
\label{finalconj}
\end{conjecture}

Thus, to prove our main result (Theorem \ref{thmforneq3}), we prove the following theorem.

\begin{theorem}
Let $F$ be of degree $3$, $[\mathfrak{b}] \in G_\mathfrak{f}$ and $\mathcal{D}$ a Shintani domain. Then, 
\[ l_\mathfrak{p}(u_{\mathfrak{p}, \lambda}(\mathfrak{b},  \mathcal{D}))= \pm c_{l_\mathfrak{p}} \cap (\omega_{\mathfrak{f}, \mathfrak{b}, \lambda}^\mathfrak{p} \cap \vartheta^\prime), \quad o_\mathfrak{p}(u_{\mathfrak{p}, \lambda}(\mathfrak{b},  \mathcal{D})) = \pm c_{o_\mathfrak{p}} \cap (\omega_{\mathfrak{f}, \mathfrak{b}, \lambda}^\mathfrak{p} \cap \vartheta^\prime). \]
\label{finalprop}
\end{theorem}

As indicated above, we are only able to show Conjecture \ref{finalconj} when $F$ is of degree $3$. Yet, the first step required in the proof can be done for $F$ of any degree. Thus, in \S 6.1 we keep $F$ of degree $n$.

\begin{remark}
If we take $g=\text{id}:F_\mathfrak{p}^\ast \rightarrow F_\mathfrak{p}^\ast$ in Definition \ref{1cocycle}, then Dasgupta-Spie\ss \ conjectured that 
\[ c_{\text{id}} \cap (\omega_{\mathfrak{f}, \mathfrak{b}, \lambda}^\mathfrak{p} \cap \vartheta^\prime) = \sigma_\mathfrak{b}(u_\lambda), \]
i.e., that we have a cohomological formula for the Gross-Stark unit.
\end{remark}

Having shown our main result, Theorem \ref{thmforneq3}. We show how the methods we have developed in fact allow us to show the following stronger result.

\begin{theorem}
Let $F$ be a totally real field of degree $3$. Then,
\[ c_{\text{id}} \cap (\omega_{\mathfrak{f}, \mathfrak{b}, \lambda}^\mathfrak{p} \cap \vartheta^\prime) = u_{\mathfrak{p}, \lambda}(\mathfrak{b},  \mathcal{D}), \]
where $\mathcal{D}$ is as in Theorem \ref{thmforneq3}.
\label{thmforequality}
\end{theorem}

This combined with the recent work of Dasgupta-Kakde in \cite{intgrossstark} gives the following corollary.

\begin{corollary}
Let $F$ be a totally real field of degree $3$ and suppose that condition (\ref{assumptionforDK}) holds. Then,
\[ c_{\text{id}} \cap (\omega_{\mathfrak{f}, \mathfrak{b}, \lambda}^\mathfrak{p} \cap \vartheta^\prime) = \sigma_\mathfrak{b}(u_\lambda), \]
up to multiplication by a root of unity.
\label{corofcohmform}
\end{corollary}

\begin{proof}
We apply Theorem 1.6 in \cite{intgrossstark} to Theorem \ref{thmforequality}.
\end{proof}

\begin{remark}
Though it is not clear at this stage, the full strength of Theorem \ref{thmforequality} is not required to show Corollary \ref{corofcohmform}. In fact, Corollary \ref{corofcohmform} will follow from our proof of Theorem \ref{finalprop}. However, since we are able to show Theorem \ref{thmforequality} we have used this in the proof of Corollary \ref{corofcohmform}.
\end{remark}

\subsection{Reduction of the Shintani domain}

In this section we let $F$ be of degree $n>1$. One of the difficulties in proving Theorem \ref{finalprop} is being able to choose a suitably nice Shintani set to work with. We do not have a Shintani domain as the Shintani set we work with is a fundamental domain for a free finite index subgroup of $E_+(\mathfrak{f})$ rather than for the full $E_+(\mathfrak{f})$. In (Lemma 2.1, \cite{MR922806}), Colmez showed that it is possible to find units $g_1, \dots ,g_{n-1} \in E_+(\mathfrak{f})$ which can be used in Lemma \ref{colmezlemma} to give a Colmez domain for $\langle g_1, \dots, g_{n-1} \rangle$. However, this choice does not give us enough control over the domain for our calculations. The main work of this paper is making a more precise choice than is given by Colmez in the case $n=3$. This is done in $\S 6.2$. We are required to show that there exist units which keep the properties required by Colmez while also satisfying some additional necessary properties. It is worth noting that currently the methods used to make this choice do not appear to extend nicely to the case $n>3$.

In this section, we show that proving our result with a free finite index subgroup of $E_+(\mathfrak{f})$ is enough to prove the result with the full $E_+(\mathfrak{f})$. We make this idea precise below. This section provides the results which give us the freedom to choose a suitable Shintani set as mentioned above.

Let $V$ be a finite index free subgroup of $E_+(\mathfrak{f})$ of rank $n-1$. Recall that $\pi$ is totally positive, congruent $1$ modulo $\mathfrak{f}$ and such that $(\pi)= \mathfrak{p}^e$ where $e$ is the order of $\mathfrak{p}$ in $G_\mathfrak{f}$. Let $\mathcal{D}_V^\prime$ be a Shintani set which is a fundamental domain for the action of $V$ on $\mathbb{R}_+^n$. As before, we shall refer to such Shintani sets as Colmez domains. We now give the notation we use for the constructions in this case. Let $\vartheta_V^\prime \in H_n(V \oplus \langle \pi \rangle, \mathbb{Z})$  be a generator. For $x_1, \dots , x_n \in V \oplus \langle \pi \rangle $ and compact open $U \subset F_\mathfrak{p}$ we put 
\[ \nu_{\mathfrak{b}, \lambda,V}^\mathfrak{p}(x_1, \dots , x_n)(U) \coloneqq \delta(x_1, \dots , x_n) \zeta_{R, \lambda}(\mathfrak{b},\overline{C}_{e_1}(x_1, \dots , x_n),U,0). \]
As before, it follows from Theorem 2.6 of \cite{MR3351752} that $\nu_{\mathfrak{b}, \lambda,V}^\mathfrak{p}$ is a homogeneous $n-1$-cocycle on $V \oplus \langle \pi \rangle$ with values in the space of $\mathbb{Z}$-distribution on $F_\mathfrak{p}$. Hence, we obtain a class
\[ \omega_{\mathfrak{f}, \mathfrak{b}, \lambda,V}^\mathfrak{p} \coloneqq [ \nu_{\mathfrak{b}, \lambda,V}^\mathfrak{p} ] \in H^{n-1}(V \oplus \langle \pi \rangle , \text{Meas}(F_\mathfrak{p},K)).  \]
We also define
\[ u_{\mathfrak{p},\lambda}(\mathfrak{b}, \mathcal{D}_V^\prime) \coloneqq \prod_{\epsilon \in V} \epsilon^{\zeta_{R,\lambda}(\mathfrak{b}, \epsilon \mathcal{D}_V^\prime \cap \pi^{-1} \mathcal{D}_V^\prime, \mathcal{O}_\mathfrak{p},0) } \pi^{\zeta_{R,\lambda}( \mathfrak{b}, \mathcal{D}_V^\prime,\mathcal{O}_\mathfrak{p},0)} \multint_\mathbb{O} x d \nu (\mathfrak{b}, \mathcal{D}_V^\prime,x). \]
At this point we have not shown that this definition makes sense. In fact, it will not make sense for all possible fundamental domains. In Proposition \ref{firstchangeofdom} we show that for the particular choice of domain we require, the definition above is sensible. We require the following comparison result later.

\begin{proposition}
Let $\mathcal{K}$ and $\mathcal{K}^\prime$ be two Colmez domains for $V$ and $\lambda$ a prime of $F$ such that $\lambda$ is $\pi$-good for $\mathcal{K}$ and $\mathcal{K}^\prime$. If $\lambda$ is also good for $(\mathcal{K}, \mathcal{K}^\prime)$, then $u_{\mathfrak{p}, \lambda}(\mathfrak{b},  \mathcal{K}) = u_{\mathfrak{p}, \lambda}(\mathfrak{b},  \mathcal{K}^\prime)$.
\label{canchangedom}
\end{proposition}

\begin{proof}
Theorem $5.3$ of \cite{MR2420508} proves this result when $V=E_+(\mathfrak{f})$. The proof of this proposition is analogous. 
\end{proof}

The following proposition shows that to prove our main result it is good enough to work with a finite index free subgroup $V \subset E_+(\mathfrak{f})$ rather than the full group. By making a good choice of $V$, we are then able to find a suitably nice Shintani set.

\begin{proposition}
Let $\mathcal{D}$ be a Shintani domain for $E_+(\mathfrak{f})$. Let $V$ be a free, finite index, subgroup of $E_+(\mathfrak{f})$ of rank $n-1$, such that $E_+(\mathfrak{f})/ V \cong \mathbb{Z}/b_1  \times \dots \times \mathbb{Z}/b_{n-1}$ with $b_1, \dots , b_{n-1} > M$, where $M = M(\pi, g_1, \dots , g_{n-1})$ is some constant that depends on $g_1, \dots , g_{n-1}$ and $\pi$ up to multiplication by an element of $E_+(\mathfrak{f})$ which we will define later. Here, we have chosen $g_1, \dots , g_{n-1}$ to be a $\mathbb{Z}$-basis for $E_+(\mathfrak{f})$ such that $g_1^{b_1}, \dots , g_{n-1}^{b_{n-1}}$ is a $\mathbb{Z}$-basis for $V$. We now define 
\[ \mathcal{D}_V \coloneqq \bigcup_{j_1=0}^{b_1-1} \dots  \bigcup_{j_{n-1}=0}^{b_{n-1}-1} g_1^{j_1} \dots g_{n-1}^{j_{n-1}} \mathcal{D}. \]
Then, for any continuous homomorphism $g: F_\mathfrak{p}^\ast \rightarrow K$, such that $g$ is trivial of $E_+(\mathfrak{f})$, if we have
\[  g(u_{\mathfrak{p}, \lambda}(\mathfrak{b},  \mathcal{D}_V))=  c_{g} \cap (\omega_{\mathfrak{f}, \mathfrak{b}, \lambda,V}^\mathfrak{p} \cap \vartheta_V^\prime), \]
then
\[ g(u_{\mathfrak{p}, \lambda}(\mathfrak{b},  \mathcal{D}))=  c_{g} \cap (\omega_{\mathfrak{f}, \mathfrak{b}, \lambda}^\mathfrak{p} \cap \vartheta^\prime). \]
We note that both $l_\mathfrak{p}$ and $o_\mathfrak{p}$ have the property that they are trivial on $E_+(\mathfrak{f})$. It is clear that $\mathcal{D}_V$ is a Colmez domain for $V$. Furthermore, since $T$ is $\pi$-good for $\mathcal{D}$ we also have that $T$ is $\pi$-good for $\mathcal{D}_V$. Thus, our definition of $u_{\mathfrak{p}, \lambda}(\mathfrak{b},  \mathcal{D}_V)$ makes sense. 
\label{firstchangeofdom}
\end{proposition}

\begin{remark}
The proof of Proposition \ref{firstchangeofdom} builds on the work of Tsosie in \cite{tsosie2018compatibility}. We follow the strategy in his proof of Proposition 2.1.4 in \cite{tsosie2018compatibility}. When considering Dasgupta-Spie\ss's formula we follow the ideas exactly. However, when considering Dasgupta's formula, $u_{\mathfrak{p}, \lambda}(\mathfrak{b}, \mathcal{D})$, we are required to alter the proof. The reason for this is that we have found a counterexample to the statement of Lemma 2.1.3 of \cite{tsosie2018compatibility}, which is used in his proof. In the appendix, we give this counterexample explicitly. It is possible to prove Proposition \ref{firstchangeofdom} without our additional assumption that $b_1, \dots, b_{n-1}>M$ however the proof becomes more lengthy. Since our strategy is to make $V$ small enough to satisfy other properties, we do not lose anything by including this simplifying assumption.
\end{remark}

\begin{proof}[Proof of Proposition \ref{firstchangeofdom}]
With the notation as given in the statement of the proposition, it is enough show the following two equalities:
\begin{equation}
    c_{g} \cap (\omega_{\mathfrak{f}, \mathfrak{b}, \lambda,V}^\mathfrak{p} \cap \vartheta_V^\prime) = [E_+(\mathfrak{f}) : V] c_{g} \cap (\omega_{\mathfrak{f}, \mathfrak{b}, \lambda}^\mathfrak{p} \cap \vartheta^\prime),
    \label{eqn1}
\end{equation}
and 
\begin{equation}
    g(u_{\mathfrak{p}, \lambda}(\mathfrak{b},  \mathcal{D}_V)) = [E_+(\mathfrak{f}):V]g(u_{\mathfrak{p}, \lambda}(\mathfrak{b},  \mathcal{D})). 
    \label{eqn2}
\end{equation}

For the first equality, we mimic the proof of Theorem $1.5$ of \cite{MR3351752}. General properties of group cohomology (see pp. 112-114, \cite{MR1324339}) yield the following commutative diagrams.
\begin{equation}
\begin{tikzcd}[cramped ,sep=small] 
{H^{n-1}(V, \text{Meas}(F_\mathfrak{p}, K))}  & \times & {H_n(V \oplus \langle \pi \rangle , \mathbb{Z} )} \arrow[r, "\cap"] \arrow[d, "\text{cores}"] & {H_1(V \oplus \langle \pi \rangle , \text{Meas}(F_\mathfrak{p}, K))} \arrow[d, "\text{cores}"] \\
{H^{n-1}(E_+(\mathfrak{f}), \text{Meas}(F_\mathfrak{p}, K))} \arrow[u, "\text{res}"] & \times & {H_n(E_+(\mathfrak{f}) \oplus \langle \pi \rangle , \mathbb{Z} )} \arrow[r, "\cap"]           & {H_1(E_+(\mathfrak{f}) \oplus \langle \pi \rangle , \text{Meas}(F_\mathfrak{p}, K))}          
\end{tikzcd}
\label{diag1}
\end{equation}
and
\begin{equation}
\begin{tikzcd}[cramped ,sep=small] 
{H^{1}(F_\mathfrak{p}^\times, C_c(F_\mathfrak{p}, K))} \arrow[d, "\text{id}"] & \times & {H_1(V \oplus \langle \pi \rangle , \text{Meas}(F_\mathfrak{p}, K))} \arrow[r, "\cap"] \arrow[d, "\text{cores}"] & K \arrow[d, "\text{id}"] \\
{H^{1}(F_\mathfrak{p}^\times, C_c(F_\mathfrak{p}, K))}           & \times & {H_1(E_+(\mathfrak{f}) \oplus \langle \pi \rangle , \text{Meas}(F_\mathfrak{p}, K))} \arrow[r, "\cap"]           & K.         
\end{tikzcd}
    \label{diag2}
\end{equation}
By Proposition 9.5 in Section 3 of \cite{MR1324339}, we have following identities,
\begin{align*}
    \text{cores}(\vartheta_V^\prime) &= [E_+(\mathfrak{f}): V] \vartheta^\prime , \\
    \text{res}(\omega_{\mathfrak{f}, \mathfrak{b}, \lambda}^\mathfrak{p}) &= \omega_{\mathfrak{f}, \mathfrak{b}, \lambda,V}^\mathfrak{p}.
\end{align*}
Diagram (\ref{diag1}) gives the equality
\[ \omega_{\mathfrak{f}, \mathfrak{b}, \lambda}^\mathfrak{p} \cap \text{cores}(\vartheta_V^\prime) = \text{cores}( \text{res}(\omega_{\mathfrak{f}, \mathfrak{b}, \lambda}^\mathfrak{p}) \cap \vartheta_V^\prime ) . \]
The identities above then show that
\[ \omega_{\mathfrak{f}, \mathfrak{b}, \lambda}^\mathfrak{p} \cap [E_+(\mathfrak{f}): V] \vartheta^\prime = \text{cores}( \omega_{\mathfrak{f}, \mathfrak{b}, \lambda,V}^\mathfrak{p} \cap \vartheta_V^\prime ) .  \]
Applying diagram (\ref{diag2}) to the above equality gives us (\ref{eqn1}). It remains to show (\ref{eqn2}). To prove (\ref{eqn2}), we prove the stronger equality
\[ u_{\mathfrak{p}, \lambda}(\mathfrak{b},  \mathcal{D}_V) = u_{\mathfrak{p}, \lambda}(\mathfrak{b},  \mathcal{D})^{[E_+(\mathfrak{f}):V]}. \]
By a result of Colmez in \S2 of \cite{MR922806} (p. 372), we have $[E_+(\mathfrak{f}):V] \zeta_\lambda(\mathfrak{b}, \mathcal{D}, U, s)=\zeta_\lambda(\mathfrak{b}, \mathcal{D}_V, U, s) $. This immediately implies that
\[ \pi^{[E_+(\mathfrak{f}):V] \zeta_{R,\lambda}( \mathfrak{b}, \mathcal{D},\mathcal{O}_\mathfrak{p},0)} = \pi^{ \zeta_{R,\lambda}( \mathfrak{b}, \mathcal{D}_V,\mathcal{O}_\mathfrak{p},0)} \]
and
\[ \left( \multint_\mathbb{O} x d \nu (\mathfrak{b}, \mathcal{D},x) \right)^{[E_+(\mathfrak{f}):V]} = \multint_\mathbb{O} x d \nu (\mathfrak{b}, \mathcal{D}_V,x). \]

It remains to show that 
\[ \left( \prod_{\epsilon \in E_+(\mathfrak{f})} \epsilon^{\zeta_{R,\lambda}(\mathfrak{b}, \epsilon \mathcal{D}\cap \pi^{-1} \mathcal{D}, \mathcal{O}_\mathfrak{p}, 0) } \right)^{[E_+(\mathfrak{f}):V]} = \prod_{\epsilon \in V} \epsilon^{\zeta_{R,\lambda}(\mathfrak{b}, \epsilon \mathcal{D}_V\cap \pi^{-1} \mathcal{D}_V, \mathcal{O}_\mathfrak{p}, 0) }.  \]

We now consider $\pi^{-1} \mathcal{D}$. By multiplying $\pi$ by an appropriate element of $E_+(\mathfrak{f})$, we can assume 
\[ \pi^{-1} \mathcal{D} \subset \bigcup_{i_1=0}^{\alpha_1} \dots  \bigcup_{i_{n-1}=0}^{\alpha_{n-1}} g_1^{i_1} \dots  g_{n-1}^{i_{n-1}} \mathcal{D}, \]
for some $\alpha_1, \dots, \alpha_{n-1}\in \mathbb{Z}_{>1}$. If we further impose that $g_1^{-1} \dots  g_{n-1}^{-1} \pi^{-1} \mathcal{D}$ is not fully contained in the positive translates of $\mathcal{D}$ and, for each $i$, choosing the minimal $\alpha_i$, then the required element of $E_+(\mathfrak{f})$ is chosen uniquely. Since the formula is independent of the choice of $\pi$ we are allowed this assumption. Now, let $M=M(\pi, g_1, \dots , g_{n-1})= \max(\alpha_1, \dots,  \alpha_{n-1})$. Since we have assumed $b_i> M$, it is easy to see that
\[ \pi^{-1} \mathcal{D}_V \subset \bigcup_{k_1=0}^{1} \dots \bigcup_{k_{n-1}=0}^{1} g_1^{k_1 b_1} \dots  g_{n-1}^{k_{n-1} b_{n-1}} \mathcal{D}_V. \]
For ease of notation, we write, for a Shintani set $D$, $\nu(D) = \zeta_{R,\lambda}(\mathfrak{b}, D, \mathcal{O}_\mathfrak{p}, 0)$. We now calculate
\begin{equation}
    \prod_{\epsilon \in V} \epsilon^{\zeta_{R,\lambda}(\mathfrak{b}, \epsilon \mathcal{D}_V\cap \pi^{-1} \mathcal{D}_V, \mathcal{O}_\mathfrak{p}, 0) }  = \prod_{i=1}^{n-1}  g_i^{  S_i}, \ \text{where} \ S_i=b_i\left(\sum_{ k_j =0 }^{1} \right)_{j \neq i} \nu(g_i^{b_i} (\prod_{j \neq i}  g_j^{b_jk_j}) \mathcal{D}_V \cap \pi^{-1} \mathcal{D}_V)    .
    \label{eqntoprove}
\end{equation}
Here we have the notation
\[ \left(\sum_{ k_j =0 }^{1} \right)_{j \neq i} = \sum_{k_1=0}^1 \dots \sum_{k_{i-1}=0}^1 \sum_{k_{i+1}=0}^1 \dots \sum_{k_{n-1}=0}^1 .  \]
To make the notation clearer, we note that
\[ S_1 = b_1 \sum_{k_2=0}^1 \dots \sum_{k_{n-1}=0}^1 \nu(g_1^{b_1} (\prod_{j =2}^{n-1}  g_j^{b_jk_j}) \mathcal{D}_V \cap \pi^{-1} \mathcal{D}_V).  \]
Consider the power above $g_1$ in (\ref{eqntoprove}). Substituting the domain $\mathcal{D}_V = \bigcup_{j_1=0}^{b_1-1} \dots  \bigcup_{j_{n-1}=0}^{b_{n-1}-1} g_1^{j_1} \dots g_{n-1}^{j_{n-1}} \mathcal{D}$ on each side of the intersection, and expanding unions and inverting the elements on the right-hand side of the intersection we have
\[
    S_1  = b_1  \left( \sum_{k_j =0 }^{1} \right)_{j=2}^{n-1} \left( \sum_{c_l=0}^{b_l-1} \sum_{a_l=0}^{b_l-1} \right)_{l=1}^{n-1} \nu(g_1^{b_1+c_1-a_1} ( \prod_{j=2}^{n-1} g_j^{b_j k_j +c_j-a_j})\mathcal{D} \cap \pi^{-1} \mathcal{D}).
\]
Since $1-b_i \leq c_i-a_i \leq b_i -1$, it is possible to rewrite our sums and deduce that the power above $g_1$ is equal to
\[ S_1 =   b_1  \left( \sum_{k_j =0 }^{1} \right)_{j=2}^{n-1} \left( \sum_{m_l=1-b_l}^{b_l-1}  \right)_{l=1}^{n-1} \prod_{l=1}^{n-1} (b_l-\mid m_l \mid)  \nu(g_1^{b_1+m_1} (\prod_{j=2}^{n-1} g_j^{b_j k_j +m_j})\mathcal{D} \cap \pi^{-1} \mathcal{D}). \]
The terms in the sum are only non-zero when $0 \leq b_1+m_1 \leq \alpha_1 $ and for $j=2, \dots, n-1$, when
\[ \begin{cases} 0 \leq m_j \leq \alpha_j &  \text{if} \ k_j=0 \\ 0 \leq b_j + m_j \leq \alpha_j & \text{if} \ k_j=1. \end{cases} \]
We now apply this to our sums, working term by term. For the $m_1$ sum we shift the index of the summand by $b_1$. We now expand the $k_2$ sum out. For the $k_2=1$ part we shift the index of the $m_2$ sum by $b_2$. Thus, we see that the power above $g_1$ in (\ref{eqntoprove}) is equal to
\begin{multline*}
     b_1 \sum_{m_1=1}^{\alpha_1} \left( \sum_{k_j =0 }^{1} \right)_{j=3}^{n-1} \left( \sum_{m_l=1-b_l}^{b_l-1}  \right)_{l=3}^{n-1} (m_1 \prod_{l=3}^{n-1} (b_l-\mid m_l \mid) ) \\ \left( \sum_{m_2=0}^{\alpha_2}   (b_2- m_2)    + \sum_{m_2=1}^{\alpha_2}  m_2    \right) \nu(g_1^{m_1}  g_{2}^{m_{2}}(\prod_{j=2}^{n-1} g_j^{b_j k_j +m_j})\mathcal{D} \cap \pi^{-1} \mathcal{D}).
\end{multline*}
Cancelling the $m_2$ terms in the sums then gives that the power above $g_1$ in (\ref{eqntoprove}) is in fact
\[  b_1 b_2 \sum_{m_1=1}^{\alpha_1} \sum_{m_2=0}^{\alpha_2} \left( \sum_{k_j =0 }^{1} \right)_{j=3}^{n-1} \left( \sum_{m_l=1-b_l}^{b_l-1}  \right)_{l=3}^{n-1} (m_1 \prod_{l=3}^{n-1} (b_l-\mid m_l \mid) )      \nu(g_1^{m_1}  g_{2}^{m_{2}}(\prod_{j=2}^{n-1} g_j^{b_j k_j +m_j})\mathcal{D} \cap \pi^{-1} \mathcal{D}). \]
Continuing to work term by term for $j=3, \dots, n-1$, and noting that $[E_+(\mathfrak{f}):V] = b_1 \dots b_{n-1}$, we are able to deduce that
\[ S_1= [E_+(\mathfrak{f}):V] \sum_{m_1=1}^{\alpha_1} \sum_{m_2=0}^{\alpha_2} \dots \sum_{m_{n-1}=0}^{\alpha_{n-1}} m_1 \nu(g_1^{m_1} \dots g_{n-1}^{m_{n-1}}\mathcal{D} \cap \pi^{-1} \mathcal{D}). \]
Similarly, the power above $g_i$ in (\ref{eqntoprove}), for $i=2, \dots, n-1$, is equal to
\[ [E_+(\mathfrak{f}):V] \sum_{m_i=1}^{\alpha_i} \left( \sum_{m_{j}=0}^{\alpha_{j}} \right)_{j \neq i} m_i \nu(g_1^{m_1} \dots g_{n-1}^{m_{n-1}}\mathcal{D} \cap \pi^{-1} \mathcal{D}) .  \]
Thus,
\[
    \prod_{\epsilon \in V} \epsilon^{\zeta_{R,\lambda}(\mathfrak{b}, \epsilon \mathcal{D}_V\cap \pi^{-1} \mathcal{D}_V, \mathcal{O}_\mathfrak{p}, 0) }  = \left( \prod_{i=1}^{n-1}  g_i^{  S_i^\prime} \right)^{[E_+(\mathfrak{f}):V]} ,  \]
where
\[ S_i^\prime=\sum_{m_i=1}^{\alpha_i} \left( \sum_{m_{j}=0}^{\alpha_{j}} \right)_{j \neq i} m_i \nu(g_1^{m_1} \dots g_{n-1}^{m_{n-1}}\mathcal{D} \cap \pi^{-1} \mathcal{D}) . \]
It remains for us to consider the error term for $u_{\mathfrak{p}, \lambda}(\mathfrak{b}, \mathcal{D})$. We calculate
\begin{align*}
    \prod_{\epsilon \in E_+(\mathfrak{f})} \epsilon^{\zeta_{R,\lambda}(\mathfrak{b}, \epsilon \mathcal{D}\cap \pi^{-1} \mathcal{D}, \mathcal{O}_\mathfrak{p}, 0) } &= \prod_{m_1=0}^{\alpha_1} \dots \prod_{m_{n-1}=0}^{\alpha_{n-1}} (g_1^{m_1} \dots g_{n-1}^{m_{n-1}})^{\nu(g_1^{m_1} \dots g_{n-1}^{m_{n-1}} \mathcal{D} \cap \pi^{-1} \mathcal{D} )} \\
    &=  \prod_{i=1}^{n-1} g_i^{S^\prime_i}.
\end{align*}
This completes the result.
\end{proof}

\subsection{Choosing a Colmez domain}

We are required to make a good choice of our finite index free subgroup $ V \subset E_+(\mathfrak{f})$. We follow the ideas initially of Colmez in \cite{MR922806}. Here, the choice of $V$ is used to give a nice Colmez domain $\mathcal{D}_V$. However, we need to use our choice of $V$ to give us both the existence of a suitable Colmez domain $\mathcal{D}_V$, and to give us some control over the translation of $\mathcal{D}_V$. This approach was not used in \cite{tsosie2018compatibility}. Instead, they used a stronger statement (Lemma 2.1.3, \cite{tsosie2018compatibility}). However, we find a counterexample to this statement. We therefore require a new approach. It is at this stage that we need to reduce to the case when $F$ is a field of degree $3$ i.e., we assume $n=3$ from now on. Note that in this case $E_+(\mathfrak{f})$ is free of rank $2$. The main aim of this section is to prove the following proposition. We remark that currently we have not been able to prove such a proposition for $n>3$.

\begin{proposition}
Let $\pi \in F_+$ then there exists $\varepsilon_1, \varepsilon_2, \omega \in E_+(\mathfrak{f})$ such that
\begin{enumerate}[1)]
    \item $\langle \varepsilon_1, \varepsilon_2 \rangle \subseteq E_+(\mathfrak{f})$ is a finite index subgroup, free of rank 2,
    \item $\delta([\varepsilon_1 \mid \varepsilon_2 ] )= - \delta([\varepsilon_2 \mid \varepsilon_1 ] )=1$,
    \item $\delta([\varepsilon_1 \mid \omega \pi ] )= - \delta([\omega \pi \mid \varepsilon_1 ] )=\delta([\varepsilon_2 \mid \omega \pi ] )= - \delta([ \omega \pi \mid \varepsilon_2 ] ) = 1$,
    \item $\omega^{-1} \pi^{-1} \in {C}([\varepsilon_1 \mid \varepsilon_2]) \cup {C}([\varepsilon_2 \mid \varepsilon_1]) \cup C(1, \varepsilon_1 \varepsilon_2) $.
\end{enumerate}
\label{propofcolmezdom2}
\end{proposition}

Recall the definition of $\delta$ from (\ref{epsilonnotation}). The choices we make through Proposition \ref{propofcolmezdom2} allow us to form a nice Colmez domain, and in the process of choosing $\varepsilon_1, \varepsilon_2, \omega$ we also allow ourselves to have some control over the translation of $\mathcal{D}_V$. We note that the hardest part of this proposition is being able to have \textit{3)} and \textit{4)} at the same time.

First, we define
\[ \Log : \mathbb{R}_+^3 \rightarrow \mathbb{R}^3, \quad (x_1, x_2, x_3) \mapsto ( \log(x_1), \log(x_2), \log(x_3)). \]
We remark that the map $\Log$ is the Dirichlet regulator on $E_+(\mathfrak{f})$. Let $\mathcal{H} \subset \mathbb{R}^3$ be the hyperplane defined by $\text{Tr}(z)=0$. Then, $\Log(E_+(\mathfrak{f}))$ is a lattice in $\mathcal{H}$. If $z = (z_1, z_2, z_3) \in \mathbb{R}_+^3$ and $\Log(z) \in \mathbb{R}^3$ is not an element of $\mathcal{H}$, then we define the projection
\[ z_\mathcal{H} = (z_1 z_2 z_3)^{-\frac{1}{3}} \cdot z . \]
We have that $\Log( z_\mathcal{H}) \in \mathcal{H}$. Note that $z$ and $z_\mathcal{H}$ lie on the same ray in $\mathbb{R}_+^3$. For any $M>0$ and $i=0,1,2$, write $l_i(M)$ for the element of $\mathcal{H}$ which has value $M$ in the $(i+1)$ place and $-M/2$ in the other places. We endow $\mathbb{R}^3$ with the sup-norm. We denote by $B(x, r)$ the ball centred at $x$ of radius $r$.

The following lemma, which builds on Lemma 2.1 of \cite{MR922806}, allows us to find a collection of possible subsets $V=\langle \varepsilon_1, \varepsilon_2 \rangle$ such that \textit{1)}, \textit{2)} and \textit{3)} in Proposition \ref{propofcolmezdom2} hold. After the proof of this lemma, we show that if we make $V$ small enough (inside $E_+(\mathfrak{f})$), then we have the freedom to choose $\varepsilon_1$, $\varepsilon_2$ and $\omega$ such that \textit{4)} also holds. We also note that Lemma \ref{lemmatomakembig} can be proven for $F$ of any degree. To keep the notation simple, we only give the proof for $n=3$.

\begin{lemma}
There exists $R_1 >0$ such that for all $R > R_1$, $M >K_1(R)$ ($K_1(R)$ is some constant we define which depends only on $R$). We have the following: For $i=1,2$ let $g_i \in E_+(\mathfrak{f}) $ and $g_\pi \in  \pi_\mathcal{H} E_+(\mathfrak{f})$ such that $\Log(g_i) \in B(l_i(M), R) $ and $\Log ( g_\pi) \in  B(l_0(M),R)$, we have
\begin{itemize}
    \item $\langle g_1, g_2 \rangle \subseteq E_+(\mathfrak{f})$ is a finite index subgroup, free of rank 2,
    \item $\delta([g_1 \mid g_2 ] )= - \delta([g_2 \mid g_1 ] )=1$,
    \item $\delta([g_1 \mid g_\pi ] )= - \delta([g_\pi \mid g_1 ] )=\delta([g_2 \mid g_\pi ] )= - \delta([ g_\pi \mid g_2 ] ) = -1$.
\end{itemize}
\label{lemmatomakembig}
\end{lemma}

\begin{proof}
This proof largely follows the ideas of Colmez in his proof of Lemma 2.1 in \cite{MR922806}. First, note that both $\Log(E_+(\mathfrak{f})) $ and $ \Log( \pi_\mathcal{H} E_+(\mathfrak{f}))$ are lattices inside $\mathcal{H}$. There exists a constant $R_1 \coloneqq R(E_+(\mathfrak{f}), \pi)$ such that for all $M>0$ and any $r > R(E_+(\mathfrak{f}), \pi)$ there exist $g_1, g_2 \in  E_+(\mathfrak{f})$ and $g_\pi \in  \pi_\mathcal{H} E_+(\mathfrak{f})$ such that $\Log(g_i) \in B(l_i(M),r)$ for $i=1,2$ and $\Log(g_\pi ) \in B(l_0(M),r)$. The existence of $R_1$ follows from Dirichlet's Unit Theorem and, in particular, the non-vanishing of the regulator of a number field. Since the $l_i(M)$ form a basis of $\mathcal{H}$, the $\Log(g_i)$ form a free family, of finite index in $\Log(E_+(\mathfrak{f}))$, if $M$ is large enough relative to $r$, say $M>k(r)$. 

Now take $M$ satisfying:
\begin{enumerate}[i)]
    \item $M \geq 2^5 r$,
    \item $M > 2^2 \log(6)$,
    \item $M >k (r)$.
\end{enumerate}
For simplicity, let $K_1(r)= \max (2^5 r, 2^2 \log(6), k(r))$ so that we only require $M> K_1(r)$.

Let $\Delta = \det([g_1 \mid g_2])$. Put $E_i = \exp (M( 1-\frac{i-2}{2}))$ and $F_i= \exp(-M(\frac{i-1}{2}))$. Hence, the matrix given by $[g_1 \mid g_2]$ is written
\[ \begin{pmatrix} 1 & \beta_{1,2} F_2 & \beta_{1,3}F_3 \\
1 & \beta_{2,2} E_2 & \beta_{2,3}E_3 \\
1 & \beta_{3,2} F_2 & \beta_{3,3}E_3
\end{pmatrix}, \]
where by i),
\[ e^{\frac{-M}{2^4}} < \beta_{i,j} < e^{\frac{M}{2^4}}.  \]
Expand $\Delta $ and isolate the diagonal term; using the bounds we defined previously we obtain
\[ \mid \Delta - e^{\frac{3 M}{2}} \beta_{2,2} \beta_{3,3} \mid \leq 5e^{\frac{M}{2^3}}  \]
and so 
\[ \Delta \geq e^{\frac{3M}{2}}(e^{\frac{-M}{2^3}}- 5 e^{(\frac{M}{2^3}-\frac{3M}{2})}) >0 \]
according to ii). We then show the other required sign properties in the same way.
\end{proof}

Note that if we choose $R>R_1^\prime \coloneqq \max( 1, R(E_+(\mathfrak{f}), \pi))$, then $K_1(R)= \max(2^5 R, k(R))$. The proof of Lemma \ref{lemmatomakembig} also gives the following, for all $R>R_1^\prime$ and $M>2^5R$. For $i=1,2$ let $g_i \in E_+(\mathfrak{f}) $ and $g_\pi \in  \pi_\mathcal{H} E_+(\mathfrak{f})$ such that $\Log(g_i) \in B(l_i(M), R) \neq \emptyset $ and $\Log ( g_\pi) \in  B(l_0(M),R) \neq \emptyset $, then
\begin{itemize}
    \item $\delta([g_1 \mid g_2 ] )= - \delta([g_2 \mid g_1 ] )=1$,
    \item $\delta([g_1 \mid g_\pi ] )= - \delta([g_\pi \mid g_1 ] )=\delta([g_2 \mid g_\pi ] )= - \delta([ g_\pi \mid g_2 ] ) = -1$.
\end{itemize}
I.e., we only lose the condition that the group, generated by $g_1, g_2$, is free of rank $2$. For later use we let $K_1^\prime(R)=2^5 R$.

We need to define a projection that depends on elements $g_1, g_2 \in E_+(\mathfrak{f})$ that generate a free group of rank $2$ and acts on $(\mathbb{R}^3_+ / \sim )$. Here, $x \sim y$ if $\exists \gamma \in \mathbb{R}_+$ such that $x=\gamma y$. We define below $\varphi_{(g_1, g_2)} : (\mathbb{R}^3_+/ \sim ) \rightarrow \mathbb{R}^2$ such that
\begin{enumerate}[i)]
    \item $\varphi_{(g_1, g_2)}(g_1) =(1,0) $ and $\varphi_{(g_1, g_2)}(g_2) =(0,1)$,
    \item for $\alpha, \beta \in \mathbb{R}_+^3$, $\varphi_{(g_1, g_2)}(\alpha \beta) =\varphi_{(g_1, g_2)}(\alpha) + \varphi_{(g_1, g_2)}(\beta)$.
\end{enumerate}
Write $g_1=(g_1(1), g_1(2), g_1(3))$ and $g_2=(g_2(1), g_2(2), g_2(3))$. If $\alpha \in \mathbb{R}_+^3/ \sim$ and $\alpha_\mathcal{H} = (\alpha_{\mathcal{H},1}, \alpha_{\mathcal{H},2}, \\ \alpha_{\mathcal{H},3})$, we define
\begin{multline}
    \varphi_{(g_1, g_2)}(\alpha) \coloneqq \left( \frac{\log(\alpha_{\mathcal{H},2}) \log(g_2(1)) - \log(\alpha_{\mathcal{H},1})\log(g_2(2))}{\log(g_2(1)) \log(g_1(2)) - \log(g_2(2))\log(g_1(1))} , \right. \\ \left. \frac{\log(\alpha_{\mathcal{H},2}) \log(g_1(1)) - \log(\alpha_{\mathcal{H},1})\log(g_1(2))}{\log(g_1(1)) \log(g_2(2)) - \log(g_1(2))\log(g_2(1))} \right).
    \label{projdefn}
\end{multline}
Choosing $\langle g_1, g_2 \rangle \subseteq E_+(\mathfrak{f})$ to be of finite index, combined with Dirichlet's Unit Theorem, gives that the denominators in (\ref{projdefn}) are non-zero and the terms are therefore well defined. This is equivalent to the fact that $\{ \Log(g_1),\Log( g_2) \}$ is a basis for $\mathcal{H}$ over $\mathbb{R}$. The idea for the function $\varphi_{(g_1, g_2)}$ comes from the following: We take $\Log(\alpha)$ and then project onto the hyperplane $\mathcal{H}$ (this is the same as choosing $\alpha_\mathcal{H}$), we then write the element of $\mathcal{H}$ in terms of the basis $\{ \Log(g_1),\Log( g_2) \}$. It is clear from the definition that we have the properties i) and ii) as required.

Now consider $g_1, g_2 \in E_+(\mathfrak{f})$ that satisfy the first two properties of Lemma \ref{lemmatomakembig}. We define
\begin{equation}
    D(g_1, g_2)= \overline{C}_{e_1}([g_1 \mid g_2]) \cup \overline{C}_{e_1}([g_2 \mid g_1]).
    \label{Deqn}
\end{equation}
Since we assume $g_1, g_2$ satisfy the second property of Lemma \ref{lemmatomakembig}, Lemma \ref{colmezlemma} gives that $D(g_1, g_2)$ is a Colmez domain for $\langle g_1, g_2 \rangle $. Additionally, we let $\overline{D}(g_1, g_2)$ be the union of $C([g_1 \mid g_2]) \cup C([g_2 \mid g_1])$ with all of their boundary cones. Then, $D(g_1, g_2) \subset \overline{D}(g_1, g_2)$ and they only differ on some of the boundary cones. Consider $\varphi_{(g_1, g_2)}(\overline{D}(g_1, g_2))$. Write
\begin{align*}
    \mathcal{C}_1(g_1, g_2) &= \varphi_{(g_1, g_2)}(C(1, g_1) \cup C(1) \cup C(g_1)), \\
    \mathcal{C}_2(g_1, g_2) &= \varphi_{(g_1, g_2)}(C(1, g_2) \cup C(1) \cup C(g_2)).
\end{align*}
Thus, $\varphi_{(g_1, g_2)}(\overline{D}(g_1, g_2))$ is bounded by $\mathcal{C}_1 \cup \mathcal{C}_2 \cup ((0,1)+\mathcal{C}_1) \cup ((1,0)+\mathcal{C}_2)$. We note that $\mathcal{C}_1$ and $\mathcal{C}_2$ are smooth lines in $\mathbb{R}^2$ with an increasing or decreasing derivative. Our next aim is to calculate the derivatives of $\mathcal{C}_1$ and $\mathcal{C}_2$ at their endpoints. For $i=1,2$ and $t \in [0,1]$, let $L_i(t)$ be the line from $(1,1,1)$ to $(g_i(1), g_i(2), g_i(3))$. We now calculate the projection of the line $L_i(t)$ under the map $z \mapsto z_\mathcal{H}$. Explicitly, we have, for $t \in [0,1]$,
\begin{multline*}
    L_i(t)_\mathcal{H} = \left( \left( \frac{(1+t(g_i(1)-1))^2}{(1+t(g_i(2)-1))(1+t(g_i(3)-1))} \right)^\frac{1}{3} , \right. \\ \left( \frac{(1+t(g_i(2)-1))^2}{(1+t(g_i(1)-1))(1+t(g_i(3)-1))} \right)^\frac{1}{3} , \\ \left. \left( \frac{(1+t(g_i(3)-1))^2}{(1+t(g_i(1)-1))(1+t(g_i(2)-1))} \right)^\frac{1}{3} \right) .
\end{multline*}
All the terms in brackets lie in $\mathbb{R}$. When we take the cube root, we are choosing $1$ as the root of unity so that $L_i(t) \in \mathbb{R}^3$. We define $\mathcal{C}_i(t) = \varphi_{(g_1, g_2)}(L_i(t))=(x_i(t), y_i(t))$ and using our formula for $L_i(t)_\mathcal{H}$, we calculate
\begin{align*}
    x_i(t) &= \frac{\log \left( \frac{(1+t(g_i(2)-1))^2}{(1+t(g_i(1)-1))(1+t(g_i(3)-1))} \right) \log(g_2(1))-\log \left( \frac{(1+t(g_i(1)-1))^2}{(1+t(g_i(2)-1))(1+t(g_i(3)-1))} \right) \log (g_2(2))  }{3 (\log(g_2(1)) \log(g_1(2)) - \log(g_2(2))\log(g_1(1)))}, \\
    y_i(t) &= \frac{\log \left( \frac{(1+t(g_i(2)-1))^2}{(1+t(g_i(1)-1))(1+t(g_i(3)-1))} \right) \log(g_1(1))-\log \left( \frac{(1+t(g_i(1)-1))^2}{(1+t(g_i(2)-1))(1+t(g_i(3)-1))} \right) \log (g_1(2))  }{3 (\log(g_1(1)) \log(g_2(2)) - \log(g_1(2))\log(g_2(1)))}.
\end{align*}
Let $l \geq 1$ be an integer. For $i=1,2$ and $t \in [0,1]$, let $L_{i,l}(t)$ be the line from $(1,1,1)$ to $(g_i(1)^l, g_i(2)^l, g_i(3)^l)$. Similar to before, we write $\mathcal{C}_{i,l}(t) = \varphi_{(g_1, g_2)}(L_{i,l}(t))=(x_{i,l}(t), y_{i,l}(t))$. We calculate $\frac{d y_{i,l}(t)}{dx_{i,l}(t)}(t=0)$ and $\frac{d y_{i,l}(t)}{dx_{i,l}(t)}(t=1)$ for $i=1,2$ and $l \geq 1$.

\begin{lemma}
We have
\[ \frac{d y_{i,l}(t)}{dx_{i,l}(t)}(t=0) =  (-1)\frac{(2g_i(2)^l-g_i(1)^l-g_i(3)^l) \log (g_1(1)) - (2g_i(1)^l-g_i(2)^l-g_i(3)^l) \log (g_1(2)) }{(2g_i(2)^l-g_i(1)^l-g_i(3)^l) \log (g_2(1)) - (2g_i(1)^l-g_i(2)^l-g_i(3)^l) \log (g_2(2))}, \]
and
\begin{multline*}
    \frac{d y_{i,l}(t)}{dx_{i,l}(t)}(t=1) = \\ (-1)\frac{(2g_i(2)^{-l}-g_i(1)^{-l}-g_i(3)^{-l}) \log (g_1(1)) - (2g_i(1)^{-l}-g_i(2)^{-l}-g_i(3)^{-l}) \log (g_1(2)) }{(2g_i(2)^{-l}-g_i(1)^{-l}-g_i(3)^{-l}) \log (g_2(1)) - (2g_i(1)^{-l}-g_i(2)^{-l}-g_i(3)^{-l}) \log (g_2(2))}.
\end{multline*}
\end{lemma}

\begin{proof}
The calculation is long but straightforward. L'H\^opital's rule is required in both calculations.
\end{proof}

In Lemma \ref{derivativeinlimit}, we show that under conditions on the units $g_1, g_2$, we have some control over the derivatives of the curves $\mathcal{C}_{1,l}(t)$ and $\mathcal{C}_{2,l}(t)$ at $t=0$ and $t=1$ for large enough $l$. We then show in Lemma \ref{fixgi} that there exist units as in Lemma \ref{lemmatomakembig} which satisfy these conditions.

\begin{lemma}
Let $g_1, g_2$ be as above. Assume further that
\begin{itemize}
    \item $g_1(2)> g_1(1)^{-2}>g_1(1)^{-1}>1 $ and,
    \item $g_2(1)< g_2(2)<1$.
\end{itemize}
Then, we have the limits
\begin{enumerate}[1)]
    \item \[ \lim_{l \rightarrow \infty} \frac{d y_{1,l}(t)}{dx_{1,l}(t)}(t=0) = (-1)\frac{2\log(g_1(1))+\log(g_1(2))}{2\log(g_2(1))+\log(g_2(2))} >0, \]
    \item \[ \lim_{l \rightarrow \infty} \frac{d y_{1,l}(t)}{dx_{1,l}(t)}(t=1) = (-1)\frac{-\log(g_1(1))+\log(g_1(2))}{-\log(g_2(1))+\log(g_2(2))} <0, \]
    \item \[ \lim_{l \rightarrow \infty} \frac{d y_{2,l}(t)}{dx_{2,l}(t)}(t=0) = (-1)\frac{-\log(g_1(1))+ \log(g_1(2))}{-\log(g_2(1))+ \log(g_2(2))} <0, \]
    \item \[ \lim_{l \rightarrow \infty} \frac{d y_{2,l}(t)}{dx_{2,l}(t)}(t=1) = (-1)\frac{ \log(g_1(1))+ 2 \log(g_1(2))}{ \log(g_2(1))+2 \log(g_2(2))} >0 . \]
\end{enumerate}
\label{derivativeinlimit}
\end{lemma}

\begin{proof}
We first note that since $g_1, g_2 \in E_+(\mathfrak{f})$ we have $g_i(3)=g_i(1)^{-1}g_i(2)^{-1}$. We work with each statement individually. Considering \textit{1)}, we have
\begin{multline*}
    \lim_{l \rightarrow \infty} \frac{d y_{1,l}(t)}{dx_{1,l}(t)}(t=0) =  \lim_{l \rightarrow \infty} (-1) \\ \frac{(2g_1(2)^l-g_1(1)^l-g_1(1)^{-l}g_1(2)^{-l}) \log (g_1(1)) - (2g_1(1)^l-g_1(2)^l-g_1(1)^{-l}g_1(2)^{-l}) \log (g_1(2)) }{(2g_1(2)^l-g_1(1)^l-g_1(1)^{-l}g_1(2)^{-l}) \log (g_2(1)) - (2g_1(1)^l-g_1(2)^l-g_1(1)^{-l}g_1(2)^{-l}) \log (g_2(2))}.
\end{multline*}
Dividing the numerator and denominator by $g_1(2)^l$, we see that
\begin{multline*}
    \lim_{l \rightarrow \infty} \frac{d y_{1,l}(t)}{dx_{1,l}(t)}(t=0) = \\  \lim_{l \rightarrow \infty} (-1)  \frac{(2- \left( \frac{g_1(1)}{g_1(2)} \right)^l - \left( \frac{g_1(1)^{-1}}{g_1(2)^{2}} \right)^l) \log (g_1(1)) - (2\left( \frac{g_1(1)}{g_1(2)} \right)^l-1-\left( \frac{g_1(1)^{-1}}{g_1(2)^{2}} \right)^l) \log (g_1(2)) }{(2- \left( \frac{g_1(1)}{g_1(2)} \right)^l - \left( \frac{g_1(1)^{-1}}{g_1(2)^{2}} \right)^l) \log (g_2(1)) - (2\left( \frac{g_1(1)}{g_1(2)} \right)^l-1-\left( \frac{g_1(1)^{-1}}{g_1(2)^{2}} \right)^l) \log (g_2(2))}.
\end{multline*}
Since $g_1(2)> g_1(1)^{-2}>g_1(1)^{-1}>1 $, the fractions $\left( \frac{g_1(1)}{g_1(2)} \right)^l, \left( \frac{g_1(1)^{-1}}{g_1(2)^{2}} \right)^l \rightarrow 0 $. Hence, 
\[ \lim_{l \rightarrow \infty} \frac{d y_{1,l}(t)}{dx_{1,l}(t)}(t=0) = (-1)\frac{2\log(g_1(1))+\log(g_1(2))}{2\log(g_2(1))+\log(g_2(2))}. \]
This value is greater than $0$ as, from the conditions we assume, $2\log(g_1(1))+\log(g_1(2)) >0$ and $2\log(g_2(1))+\log(g_2(2))<0$ thus giving \textit{1)}.

For \textit{2)}, we have 
\begin{multline*}
    \lim_{l \rightarrow \infty} \frac{d y_{1,l}(t)}{dx_{1,l}(t)}(t=1) =  \lim_{l \rightarrow \infty} (-1) \\ \frac{(2g_1(2)^{-l}-g_1(1)^{-l}-g_1(1)^{l}g_1(2)^{l}) \log (g_1(1)) - (2g_1(1)^{-l}-g_1(2)^{-l}-g_1(1)^{l}g_1(2)^{l}) \log (g_1(2)) }{(2g_1(2)^{-l}-g_1(1)^{-l}-g_1(1)^{l}g_1(2)^{l}) \log (g_2(1)) - (2g_1(1)^{-l}-g_1(2)^{-l}-g_1(1)^{l}g_1(2)^{l}) \log (g_2(2))}.
\end{multline*}
Multiplying the numerator and denominator by $g_1(1)^{-l}g_1(2)^{-l}$, we see that
\begin{multline*}
    \lim_{l \rightarrow \infty} \frac{d y_{1,l}(t)}{dx_{1,l}(t)}(t=1) =  \lim_{l \rightarrow \infty} (-1) \\ \frac{(2 \left( \frac{g_1(1)^{-1}}{g_1(2)^2} \right)^l -  \left( \frac{g_1(1)^{-2}}{g_1(2)} \right)^l -1)  \log (g_1(1)) - (2\left( \frac{g_1(1)^{-2}}{g_1(2)} \right)^l-\left( \frac{g_1(1)^{-1}}{g_1(2)^2} \right)^l-1) \log (g_1(2)) }{(2 \left( \frac{g_1(1)^{-1}}{g_1(2)^2} \right)^l -  \left( \frac{g_1(1)^{-2}}{g_1(2)} \right)^l -1) \log (g_2(1)) - (2\left( \frac{g_1(1)^{-2}}{g_1(2)} \right)^l-\left( \frac{g_1(1)^{-1}}{g_1(2)^2} \right)^l-1) \log (g_2(2))}.
\end{multline*}
Since $g_1(2)> g_1(1)^{-2}>g_1(1)^{-1}>1 $, the fractions $\left( \frac{g_1(1)^{-1}}{g_1(2)^2} \right)^l, \left( \frac{g_1(1)^{-2}}{g_1(2)} \right)^l \rightarrow 0$. Hence,
\[ \lim_{l \rightarrow \infty} \frac{d y_{1,l}(t)}{dx_{1,l}(t)}(t=1) = (-1)\frac{-\log(g_1(1))+\log(g_1(2))}{-\log(g_2(1))+\log(g_2(2))}. \]
From the conditions we assume, $-\log(g_1(1))+\log(g_1(2))>0$ and $-\log(g_2(1))+\log(g_2(2)) >0$. Hence, we get the correct sign.

For \textit{3)}, consider $\lim_{l \rightarrow \infty} \frac{d y_{2,l}(t)}{dx_{2,l}(t)}(t=0)$ and multiply the numerator and denominator of the corresponding fraction by $g_2(1)^lg_2(2)^l$. Since $g_2(1)^l, g_2(2)^l \rightarrow 0$, we see that
\[ \lim_{l \rightarrow \infty} \frac{d y_{2,l}(t)}{dx_{2,l}(t)}(t=0) = (-1)\frac{-\log(g_1(1))+\log(g_1(2))}{-\log(g_2(1))+\log(g_2(2))}. \]
From the conditions we assume, $-\log(g_1(1))+\log(g_1(2))>0$ and $-\log(g_2(1))+\log(g_2(2)) >0$. Hence, we get the correct sign.

For \textit{4)}, consider $\lim_{l \rightarrow \infty} \frac{d y_{2,l}(t)}{dx_{2,l}(t)}(t=1)$ and multiply the numerator and denominator of the corresponding fraction by $g_2(1)^l$. Since $g_2(1)^l, g_2(2)^l \rightarrow 0$, we see that
\[ \lim_{l \rightarrow \infty} \frac{d y_{2,l}(t)}{dx_{2,l}(t)}(t=1) = (-1)\frac{-\log(g_1(1))- 2 \log(g_1(2))}{-\log(g_2(1))- 2 \log(g_2(2))}= (-1)\frac{\log(g_1(1))+ 2 \log(g_1(2))}{\log(g_2(1))+ 2 \log(g_2(2))} . \]
From the conditions we assume, $\log(g_1(1))+2\log(g_1(2))>0$ and $\log(g_2(1))+2\log(g_2(2)) <0$. Hence, we get the correct sign.

\end{proof}

We now show that it is possible to find elements that satisfy the properties in the statement of Lemma \ref{derivativeinlimit}. Note that in Lemma \ref{fixgi} we do not show that $g_1, g_2$ generate a finite index subgroup in $E_+(\mathfrak{f})$. After the proof of the lemma, we choose $r$ and $M$ to be large enough so that the conditions of Lemma \ref{lemmatomakembig} are satisfied as well.

\begin{lemma}
There exists $R_2>0 $ such that for all $R>R_2$, $M>K_2(R)$ ($K_2(R)$ is some constant we define which depends only on $R$). We have the following: For $i=1,2$, there exists $g_i \in E_+(\mathfrak{f})$ such that $ \Log(g_i) \in B(l_i(M),R)$ and if we write $g_i=(g_i(1), g_i(2), g_i(3))$, 
\begin{enumerate}[i)]
    \item $g_1(2)> g_1(1)^{-2}>g_1(1)^{-1}>1 $,
    \item $g_2(1)< g_2(2)<1$.
\end{enumerate}
\label{fixgi}
\end{lemma}

\begin{proof}
We only give the proof for $g_1$ since the proof for $g_2$ is similar and easier. Recall that $l_1(M)=(-M/2, M, -M/2)$. Since $\Log(E_+(\mathfrak{f}))$ is a lattice inside $\mathcal{H}$, we are able to fix $R_2>0$ such that if $R>R_2$ then for all $M>0$ there exists $x=(x_1, x_2, x_3) \in E_+(\mathfrak{f})$ such that
\begin{itemize}
    \item $\Log(x) \in B(l_1(M), R)$,
    \item $\log(x_1)+\frac{M}{2}>0$,
    \item $\log(x_2)-M>0$.
\end{itemize} 
Such a choice is possible since $\Log(E_+(\mathfrak{f}))$ is a lattice in $\mathcal{H}$. We let $K_2(R)=2R$ and impose that $M>K_2(R)$. With this assumption we then have, in addition to the properties above, $\log(x_1)<0$. The result now follows by noting that \textit{i)} is equivalent to 
\begin{enumerate}[$\textit{i}^\prime)$]
    \item $\log(g_1(2))> -2 \log( g_1(1))> - \log( g_1(1))>0 $.
\end{enumerate}
\end{proof}

We fix $r> \max(R_1^\prime, R_2,1)$ and $M_1> \max(K_1(r),K_2(r), 4K_1^\prime(r))$. We choose $g_1, g_2 \in E_+(\mathfrak{f})$ such that, for $i=1,2$, $\Log(g_i) \in B(l_i(M_1), r)$ and satisfies \textit{i)} and \textit{ii)} in the statement of Lemma \ref{fixgi}, respectively. We remark that the reason for taking $4K_1^\prime(r)$ rather than simply $K_1^\prime(r)$ will not be apparent until Lemma \ref{lemmafindpi-1}. The choices we make here are henceforth fixed. For clarity, we note that under these conditions we have, by Lemma \ref{lemmatomakembig} and Lemma \ref{fixgi}, the existence of $g_1, g_2 \in E_+(\mathfrak{f})$ such that 
\begin{itemize}
    \item $\langle g_1, g_2 \rangle \subseteq E_+(\mathfrak{f})$ is a finite index subgroup, free of rank 2,
    \item $\delta([g_1 \mid g_2 ] )= - \delta([g_2 \mid g_1 ] )=1$,
    \item $g_1(2)> g_1(1)^{-2}>g_1(1)^{-1}>1 $,
    \item $g_2(1)< g_2(2)<1$.
\end{itemize}
We fix this choice of $g_1$ and $g_2$ for the remainder of the paper. We now show that when choosing our subgroup $V$, we are allowed to raise our current choices to positive powers. This enables us to make use of the controls we obtained in Lemma \ref{derivativeinlimit}.

\begin{proposition}
For all $l \geq 1$, we have
\begin{enumerate}[1)]
    \item $\langle g_1^l, g_2^l \rangle \subseteq E_+(\mathfrak{f})$ is a finite index subgroup, free of rank 2,
    \item $\delta([g_1^l \mid g_2^l ] )= - \delta([g_2^l \mid g_1^l ] )=1$.
\end{enumerate}
\label{pffor12}
\end{proposition}

\begin{proof}
Since $\langle g_1, g_2 \rangle $ is free of rank $2$ and finite index, we must also have that $\langle g_1^l, g_2^l \rangle $ is also free of rank $2$ and finite index. Let $i=1,2$, since $\Log(g_i) \in B(l_i(M_1), r)$, we have $\Log(g_i^l) \in B(l_i(M_1l), rl)$. Thus, $rl \geq r> R_1^\prime  $ and $lM_1>2^5rl$. By the paragraph following the proof of Lemma \ref{lemmatomakembig}, we therefore get that \textit{2)} holds as well.
\end{proof}

We are now able to use our choices to control the curves $\mathcal{C}_{1,l}(t)$ and $\mathcal{C}_{2,l}(t)$.
 
\begin{corollary}
There exists $L_1>0$ such that for all $l >L_1$
\begin{enumerate}[i)]
    \item $y_{1,l}(t) \geq 0$,
    \item $x_{2,l}(t) \leq 0$,
    \item $ 0 \leq x_{1,l}(t) \leq l $,
    \item $ 0 \leq y_{2,l}(t) \leq l $,
\end{enumerate}
for all $t \in [0,1]$.
\label{directionofsides}
\end{corollary}

\begin{proof}
By Lemma \ref{derivativeinlimit}, there exists $L_1>0$ such that for all $l>L_1$
\begin{align*}
    \frac{d y_{1,l}(t)}{dx_{1,l}(t)}(t=0) >0, \quad \frac{d y_{1,l}(t)}{dx_{1,l}(t)} (t=1) &<0,  \\
    \frac{d y_{2,l}(t)}{dx_{2,l}(t)}(t=0) <0, \quad
    \frac{d y_{2,l}(t)}{dx_{2,l}(t)}(t=1) &>0.
\end{align*}
We recall the definition of $D(g_1^l, g_2^l)$ from equation (\ref{Deqn}) and note that from \textit{2)} in Proposition \ref{pffor12} we have the sign properties required to show that $D(g_1^l, g_2^l)$ forms a fundamental domain for the action of $\langle g_1^l, g_2^l \rangle $ on $\mathbb{R}_+^3$. This follows from Lemma 2.2 of \cite{MR922806}. From this we deduce two key properties. Firstly, we have
\[ \mathcal{C}_{1,l} \cap ((0,l)+\mathcal{C}_{1,l}) = \emptyset \quad \text{and} \quad \mathcal{C}_{2,l} \cap ((l,0)+\mathcal{C}_{2,l})= \emptyset. \]
Secondly, the curves $\mathcal{C}_{1,l}$ and $\mathcal{C}_{2,l}$ can only intersect at the endpoints. More precisely, we have 
\begin{align*}
    \mathcal{C}_{1,l} \cap \mathcal{C}_{2,l} &=\{ (0,0) \} , \\
    ((0,l)+ \mathcal{C}_{1,l}) \cap \mathcal{C}_{2,l} &= \{(0, l)\} , \\
    \mathcal{C}_{1,l} \cap ((l,0)+ \mathcal{C}_{2,l}) &= \{(l,0) \} , \\
    ((0,l)+ \mathcal{C}_{1,l}) \cap ((l,0)+ \mathcal{C}_{2,l})&= \{ (l,l) \}.
\end{align*}

Henceforth, we choose $l>L_1$. Note that since the map $\varphi_{(g_1, g_2)}$ is equivalent to taking a projection followed by the $\Log$ map, followed by a base change, it maps straight lines in $\mathbb{R}_+^{3}$, which are not contained in rays, to continuous strictly convex curves in $\mathbb{R}^2$. We have strictly convex curves as we can never obtain straight lines in $\mathbb{R}^2$ from straight lines in $\mathbb{R}^3_+$ that are not contained in rays. More precisely, let $\gamma(t)$, $t\in [0,1]$, be any straight line of finite length in $\mathbb{R}^3_+$ where $\gamma(0)$ and $\gamma(1)$ are not both lying on the same ray. Then, we have 
\begin{multline*}
    \{ \varphi_{(g_1, g_2)}(\gamma(0)) +k (\varphi_{(g_1, g_2)}(\gamma(1))-\varphi_{(g_1, g_2)}(\gamma(0))) \mid k \in [0,1] \} \cap \{ \varphi_{(g_1, g_2)}(\gamma(t)) \mid t \in [0,1]   \} \\ =\{ \varphi_{(g_1, g_2)}(\gamma(0)), \varphi_{(g_1, g_2)}(\gamma(1)) \}.
\end{multline*}

We first show \textit{ii)}. Proceeding by contradiction, we suppose that $x_{2,l}(T)>0$ for some $T \in [0,1]$. Since $\mathcal{C}_{2,l}$ is strictly convex and has contains the points $(0,0)$ and $(0,l)$, we deduce that $x_{2,l}(t) \geq 0$ for all $t \in [0,1]$. Since
\[ \frac{d y_{2,l}(t)}{dx_{2,l}(t)}(t=0) <0 \quad \text{and} \quad \frac{d y_{2,l}(t)}{dx_{2,l}(t)}(t=1) >0, \]
there exist $T_1, T_2 \in [0,1] $ such that $y_{2,l}(T_1)<0$ and $y_{2,l}(T_2)>l$.

Consider $\mathcal{C}_{1,l}$. Since $\frac{d y_{1,l}(t)}{dx_{1,l}(t)}(t=0) >0$ and $\mathcal{C}_{1,l}$ is strictly convex, we must have that $y_{1,l}(t) \leq 0$ for all $t \in [ 0,1]$. Note that if we had $y_{1,l}(t) > 0$ for some $t \in [0,1]$ then $\mathcal{C}_{1,l}$ and $\mathcal{C}_{2,l}$ would intersect at at least one point other than $(0,0)$.

We now consider the curve $(0,l)+\mathcal{C}_{1,l}$. Since we have $\mathcal{C}_{2,l} \cap ((0,l)+\mathcal{C}_{1,l}) = \{ (0,l) \} $, $y_{1,l}(t) \leq 0$ for all $t \in [ 0,1]$ and the existence of $T_1$, there exists $K \in [0,1]$ such that
\begin{itemize}
    \item $l+ y_{1,l}(K) <0$,
    \item $x_{1,l}(K)=0$, and
    \item $x_{1,l}(t) \leq 0$ for all $t \in [0,K]$.
\end{itemize}
These three conditions imply that $\mathcal{C}_{1,l} \cap ((0,l)+ \mathcal{C}_{1,l}) \neq \emptyset$ which is a contradiction. This gives a contradiction to the existence of $T \in [0,1]$ such that $x_{2,l}(T)>0$. Hence, we have that $x_{2,l}(t)\leq 0$ for all $t \in [0,1]$ and so \textit{ii)} holds. 

To prove \textit{i)} we again work by contradiction and suppose that $y_{1,l}(T) < 0$ for some $T \in [0,1]$. As before, we deduce that $y_{1,l}(t) \leq 0$ for all $t \in [0,1]$. Since 
\[\frac{d y_{1,l}(t)}{dx_{1,l}(t)}(t=0) >0 \quad \text{and} \quad \frac{d y_{1,l}(t)}{dx_{1,l}(t)} (t=1) <0, \]
there exist $T_1, T_2 \in [0,1] $ such that $x_{1,l}(T_1)<0$ and $x_{1,l}(T_2)>l$. As before, we consider the curve $(0,l)+\mathcal{C}_{1,l}$. Using a similar argument as above, we are able to show that $\mathcal{C}_{1,l} \cap ((0,l)+ \mathcal{C}_{1,l}) \neq \emptyset$. This contradiction then gives us that \textit{i)} holds.

From what we deduced about the derivatives and the fact that the first two statements hold, it is clear that \textit{iii)} and \textit{iv)} must also hold.
\end{proof}

The results of Corollary \ref{directionofsides}, combined with the fact that $\mathcal{C}_{1,l}$ and $\mathcal{C}_{2,l}$ are strictly convex curves, gives us that the image of $\mathcal{C}_{1,l} \cup \mathcal{C}_{2,l} \cup ((0,l)+ \mathcal{C}_{1,l}) \cup ((l,0)+ \mathcal{C}_{2,l}) $ is always in a similar form to the following example. Note that in the image below we choose an example where we can take $l=1$. Throughout the following proofs one should try to keep the image below in mind. We give more details on the explicit choices and calculations needed to form this image in the appendix. Although the image appears to show that the lines $\mathcal{C}_{1,l}$ and $(l,0)+\mathcal{C}_{2,l}$ overlap, this in fact does not happen. This only appears in the diagram due to the fixed thickness of the lines.

\begin{figure}[h]
    \centering
    \includegraphics[scale=0.35]{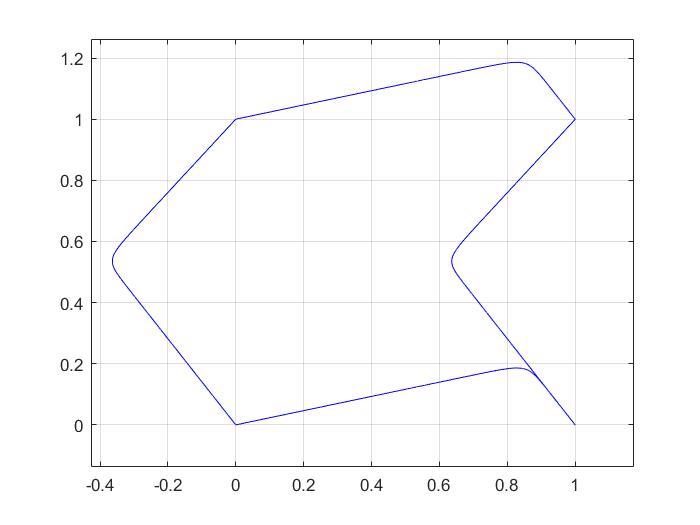}
    \caption{A Colmez domain chosen as in Corollary \ref{directionofsides}.}
    \label{fig:colmezdom}
\end{figure}

Using the corollary above, the next lemma shows that we are now able to find an element of $\pi^{-1}_\mathcal{H} E_+(\mathfrak{f})$ which satisfies properties similar to \textit{3)} and \textit{4)} of Proposition \ref{propofcolmezdom2}. Note that the element we find in the next lemma will directly give rise an element which satisfies \textit{3)} and \textit{4)} of Proposition \ref{propofcolmezdom2}.

\begin{lemma}
There exists $L_2>0$ such that for all $l>\max(L_1,L_2)$, there exists $\alpha \in \pi^{-1}_\mathcal{H} E_+(\mathfrak{f})$ such that
\begin{itemize}
    \item $\alpha \in C([g_1^l \mid g_2^l ]) \cup C([g_2^l \mid g_1^l]) \cup C(1, g_1^l g_2^l) $,
    \item $\Log(\alpha ) \in B(-l_0(lM_1), 4lr)$.
\end{itemize}
\label{lemmafindpi-1}
\end{lemma}

\begin{proof}
We assume that $l>L_1$. By Lemma \ref{derivativeinlimit}, we have the limit
\[ d_2 = \lim_{l \rightarrow \infty} \frac{d y_{2,l}(t)}{dx_{2,l}(t)}(t=1) >0. \]
Then, there exists $ L_2^\prime >0$ such that for all $l >L_2^\prime$, 
\[  \frac{d y_{2,l}(t)}{dx_{2,l}(t)}(t=1) >\frac{d_2}{2}. \]
Let $\theta = \arctan(d_2/2)>0$ and for $Q>0$ define $T(\theta, Q , (l,l))$ to be the triangle drawn below. 

\begin{center}
\begin{tikzpicture}[
  my angle/.style={
    every pic quotes/.append style={text=cyan},
    draw=cyan,
    angle radius=1cm,
  }]
  \coordinate [label=right:$ (l {,} l) $] (C) at (5,2);
  \coordinate [label=left:$(l-Q \cos (\theta) {,} l)$ ] (A) at (0,2);
  \coordinate [] (B) at (0,0);
  \draw (C) -- node[below] {$Q$} (B) -- node[right] {} (A) -- node[below] {} (C);
  \draw (A) +(.25,0) |- +(0,-.25);
  \pic [my angle, "$\theta$"] {angle=A--C--B};
\end{tikzpicture}
\end{center}

We choose $Q$ big enough such that for all $l>0$, there exists $\alpha \in \pi_\mathcal{H}^{-1} E_+(\mathfrak{f}) \cap T(\theta, Q , (l,l)) $. As seen in the proof of Lemma \ref{lemmatomakembig}, the existence of such a $Q$ follows from Dirichlet's Unit Theorem and, in particular, the non-vanishing of the regulator of a number field. The next idea of the proof is to make $l$ big enough such that the triangle $T(\theta, Q , (l,l))$ is guaranteed to be contained inside $C([g_1^l \mid g_2^l ]) \cup C([g_2^l \mid g_1^l]) \cup C(1, g_1^l g_2^l)$. The triangle is chosen such that for all $l >L_2^\prime $, it lies to the left of the curve $(l,0) + \mathcal{C}_{2,l}$. Again by Lemma \ref{derivativeinlimit}, we have the limit
\[ d_1 = \lim_{l \rightarrow \infty} \frac{d y_{1,l}(t)}{dx_{1,l}(t)}(t=1) <0. \]
Then, there exists $ L_2^{\prime \prime} >0$ such that for all $l >L_2^{\prime \prime}$, 
\[ 0> \frac{d y_{1,l}(t)}{dx_{1,l}(t)}(t=1) >\frac{d_1}{2}. \]
Note that the first inequality above follows from our assumption that $l>L_1$. Let $\gamma = - \arctan(d_1/2) \\ >0$ and  define $T(\gamma , (l,l))$ to be the triangle drawn below. 

\begin{center}
\begin{tikzpicture}[
  my angle/.style={
    every pic quotes/.append style={text=cyan},
    draw=cyan,
    angle radius=1cm,
  }]
  \coordinate [label= below:$(l{,}0)$] (C) at (5,0);
  \coordinate [label=right:$(l{,}l)$] (A) at (5,5);
  \coordinate (B) at (2,5);
  \draw (C) -- node[above] {} (B) -- node[right] {} (A) -- node[below] {} (C);
  \draw (A) +(-.25,0) |- +(0,-.25);
  \pic [my angle, "$\gamma$"] {angle=A--C--B};
\end{tikzpicture}
\end{center}

Note that for all $l> \max(L_2^\prime, L_2^{\prime \prime}) $, we have $T(\theta, Q , (l,l)) \cap T(\gamma , (l,l)) \subset C([g_1^l \mid g_2^l ]) \cup C([g_2^l \mid g_1^l]) \cup C(1, g_1^l g_2^l)  $. Since the size of $T(\theta, Q , (l,l))$ is fixed, there exists $L_2^{\prime \prime \prime}$ such that for all $l>L_2^{\prime \prime \prime}$, $T(\theta, Q , (l,l)) \subset T(\gamma , (l,l))$. Thus, if we choose $\widetilde{L_2} = \max(L_2^\prime , L_2^{\prime \prime}, L_2^{\prime \prime \prime})$, then for $l> \widetilde{L_2}$, there exists $\alpha \in \pi_\mathcal{H}^{-1} E_+(\mathfrak{f}) $ such that  $\alpha \in C([g_1^l \mid g_2^l ]) \cup C([g_2^l \mid g_1^l]) \cup C(1, g_1^l g_2^l)$. Since $\Log( g_i) \in B(l_i(M_1), r)$, we have $\Log(g_1^l g_2^l) \in B(-l_0(lM_1), 2lr)  $. The size of the triangle $T(\theta, Q , (l,l))$ is fixed and always has a point at $(l,l)$. It is therefore clear that for $l$ big enough (say $l>\widetilde{L_2}^\prime$) the pre-image of the triangle before the change of basis is contained in $B(-l_0(lM_1), 4lr)$. Note that we achieved this by simply doubling the radius of the ball. We finish by setting $L_2= \max (\widetilde{L_2}, \widetilde{L_2}^\prime)$ to ensure that we obtain all the required conditions.
\end{proof}

We are now ready to prove the proposition we stated at the start of this section. 

\begin{proof}[Proof of Proposition \ref{propofcolmezdom2}]
Let $l>\max (L_1, L_2)$, and write $\varepsilon_i=g_i^l$ for $i=1,2$. By Proposition \ref{pffor12}, we get \textit{1)} and \textit{2)} in Proposition \ref{propofcolmezdom2}. By Lemma \ref{lemmafindpi-1}, there exists $\alpha \in \pi^{-1}_\mathcal{H} E_+(\mathfrak{f})$ such that 
\begin{itemize}
    \item $\alpha \in C([g_1^l \mid g_2^l ]) \cup C([g_2^l \mid g_1^l]) \cup C(1, g_1^l g_2^l) $,
    \item $\Log(\alpha ) \in B(-l_0(lM_1), 4lr)$.
\end{itemize}
We then define $\omega = \alpha^{-1} \pi_\mathcal{H}^{-1} \in E_+(\mathfrak{f})$. Since $\alpha= \pi_\mathcal{H}^{-1} \omega^{-1} = k \cdot \pi^{-1} \omega^{-1} $ for some $k \in \mathbb{R}_{>0}$, in the second equality we consider the elements as vectors in $\mathbb{R}_+^3$. Hence, we have
\[ \omega^{-1} \pi^{-1} \in  C([g_1^l \mid g_2^l ]) \cup C([g_2^l \mid g_1^l]) \cup C(1, g_1^l g_2^l) \subset \overline{C}_{e_1}([\varepsilon_1 \mid \varepsilon_2]) \cup \overline{C}_{e_1}([\varepsilon_2 \mid \varepsilon_1]). \]
Thus, we obtain \textit{4)} of the proposition. Now, let $g_\pi = \alpha^{-1}= \pi_\mathcal{H} \omega$. Then,
\[ \Log(g_\pi) \in B(l_0(lM_1), 4lr). \]
Since $M_1 > 4K_1^\prime(r)=4 \cdot 2^5 r$, we have $lM_1>K_1^\prime(4lr)$. Thus, by Lemma \ref{lemmatomakembig}, we obtain \textit{3)}. This completes the proof of the proposition.
\end{proof}

We fix the choice of $\varepsilon_1, \varepsilon_2$ and, for ease of notation, write $\pi = \omega \pi$, as is prescribed by Proposition \ref{propofcolmezdom2}. We assume in addition to the properties given by Proposition \ref{lemmatomakembig} that $\langle \varepsilon_1, \varepsilon_2 \rangle \cong \mathbb{Z}/b_1 \mathbb{Z} \times \mathbb{Z}/b_1 \mathbb{Z} $ with $b_1, b_2 $ large enough to satisfy the conditions required in Proposition \ref{firstchangeofdom}. This is achieved by simply choosing a larger $l$ than in the proof of Proposition \ref{propofcolmezdom2}, if it is required. Let  
\[\mathcal{B} \coloneqq  \overline{C}_{e_1} ([\varepsilon_{1}  \mid  \varepsilon_{2}]) \cup \overline{C}_{e_1} ([\varepsilon_{2}  \mid  \varepsilon_{1}]) . \]
By $2)$ of Proposition \ref{propofcolmezdom2} and Lemma \ref{colmezlemma}, this is a Colmez domain for $\langle \varepsilon_1, \varepsilon_2 \rangle$. We also define
\begin{align*}
    \mathcal{B}_1 & \coloneqq  \overline{C}_{e_1} ([\varepsilon_{2}  \mid  \pi ]) \cup \overline{C}_{e_1} ([\pi \mid  \varepsilon_{2}]), \\
    \mathcal{B}_2 & \coloneqq  \overline{C}_{e_1} ([\varepsilon_{1}  \mid  \pi]) \cup \overline{C}_{e_1} ([\pi \mid  \varepsilon_{1}]).
\end{align*}
Then by \textit{3)} of Proposition \ref{propofcolmezdom2}, $\mathcal{B}_1$ is a fundamental domain for the action of $\langle \varepsilon_2, \pi \rangle$ on $\mathbb{R}_+^3$ and $\mathcal{B}_2$ is a fundamental domain for the action of $\langle \varepsilon_1, \pi \rangle$ on $\mathbb{R}_+^3$. We are now ready to show that through our choice of $\varepsilon_1, \varepsilon_2$ and $\pi$, we can obtain control over the $\pi^{-1}$ translate of $\mathcal{B}$.

\begin{proposition}
With the choice of $\pi$ fixed before, we have 
\[ \pi^{-1} \mathcal{B} \subset \bigcup_{k_1=0}^1 \bigcup_{k_2=0}^{2} \varepsilon_1^{k_1} \varepsilon_2^{k_2} \mathcal{B}.  \]
\label{whereispiinverse}
\end{proposition}

\begin{remark}
The purpose of the careful choice of $\varepsilon_1$ and $\varepsilon_2$ is to obtain this proposition. In \cite{tsosie2018compatibility}, a stronger statement than this is used (Lemma 2.1.3, \cite{tsosie2018compatibility}). However, as stated before, we obtain a counterexample to the statement of Lemma 2.1.3. This counterexample is given explicitly in the appendix.
\end{remark}

\begin{proof}[Proof of Proposition \ref{whereispiinverse}]
We show the following containments. The result follows from this.
\begin{enumerate}[i)]
    \item $\pi^{-1} C(1, \varepsilon_1)  \subset \mathcal{B} \cup \varepsilon_1 \mathcal{B} \cup \varepsilon_2 \mathcal{B} \cup \varepsilon_1 \varepsilon_2 \mathcal{B}$, 
    \item $\pi^{-1} C(1, \varepsilon_2)  \subset \mathcal{B} \cup \varepsilon_2 \mathcal{B}$.
\end{enumerate}
It is enough to show i) and ii) since there are no holes in $\bigcup_{k_1=0}^1 \bigcup_{k_2=0}^{2} \varepsilon_1^{k_1} \varepsilon_2^{k_2} \mathcal{B}$. Thus, if we can show that the boundary of $\overline{\mathcal{B}}$ lies in $\bigcup_{k_1=0}^1 \bigcup_{k_2=0}^{2} \varepsilon_1^{k_1} \varepsilon_2^{k_2} \mathcal{B}$, then we are done. The combination of i) and ii) gives us exactly this.

We begin with i). We consider the curves under our map $\varphi_{(g_1, g_2)}$. Throughout this proof we refer to the positive second coordinate as ``up'', the positive first coordinate as ``right'', and similarly for ``down'' and ``left''. Since $\pi^{-1} $ is chosen to be in the interior of $\mathcal{B}$ and by Corollary \ref{directionofsides}, we must have that $ \varphi_{(g_1, g_2)}(\pi^{-1})$ lies above $\mathcal{C}_{1,l}$ in $\mathbb{R}_2$. Since the curve $\mathcal{C}_{1,l}$ is strictly convex, as defined before, we see that the curve
\[ \varphi_{(g_1, g_2)}(\pi^{-1}) + \mathcal{C}_{1,l} \quad \text{lies above} \quad \bigcup_{k \in \mathbb{Z}} ((kl,0)+\mathcal{C}_{1,l}). \]
By \textit{2)} of Proposition \ref{propofcolmezdom2}, $\mathcal{B}$ forms a fundamental domain. From this, it follows that $\mathcal{C}_{1,l}$ must lie between $\bigcup_{k \in \mathbb{Z}} ((0,kl)+\mathcal{C}_{2,l})$ and $\bigcup_{k \in \mathbb{Z}} ((l,kl)+\mathcal{C}_{2,l})$. Hence, 
\[ \bigcup_{k \in \mathbb{Z}} ((0,kl)+\mathcal{C}_{2,l}) \ \text{is to the left of} \ \varphi_{(g_1, g_2)}(\pi^{-1}) + \mathcal{C}_{1,l} \ \text{is to the left of} \ \bigcup_{k \in \mathbb{Z}} ((2l,kl)+\mathcal{C}_{2,l}).  \]
At this point, we have shown that
\[ \pi^{-1} C(1, \varepsilon_1) \subset \bigcup_{k_2 \geq 0} \varepsilon_2^{k_2} (\mathcal{B} \cup \varepsilon_1 \mathcal{B}). \] 
Now, suppose that $\pi^{-1} C(1, \varepsilon_1) \cap \varepsilon_2^{2} (\mathcal{B} \cup \varepsilon_1 \mathcal{B}) \neq \emptyset$. Then, this means that after moving back to $\mathbb{R}^2$ we see that some point on $\mathcal{C}_{1,l}$ has $y$ value greater than $1$. Consider the cone $C(1, \pi^{-1} \varepsilon)$. By \textit{3)} of Proposition \ref{propofcolmezdom2}, we have that $\mathcal{B}_2$ is well defined, and thus $\pi^{-1} \mathcal{B}_2$ is also well defined. Hence, in $\mathbb{R}^2$ we must have that $\varphi_{(g_1, g_2)}(C(1, \pi^{-1} \varepsilon))$ is above $\mathcal{C}_{1,l}$ but also passes below $\varphi_{(g_1, g_2)}(\pi^{-1})$. Yet, since some point on $\mathcal{C}_{1,l}$ has $y$ value greater than $1$, the curve $\varphi_{(g_1, g_2)}(C(1, \pi^{-1} \varepsilon))$ cannot be strictly convex. This gives us a contradiction. Hence, i) holds.

For ii), we use similar methods as above to deduce that 
\[ \bigcup_{k \in \mathbb{Z}} ((kl,0)+\mathcal{C}_{1,l}) \quad \text{is below} \quad \varphi_{(g_1, g_2)}(\pi^{-1}) + \mathcal{C}_{2,l} \quad \text{is below} \quad \bigcup_{k \in \mathbb{Z}} ((kl,2l)+\mathcal{C}_{1,l}). \]
Using Corollary \ref{directionofsides}, we have
\[ \pi^{-1} C(1, \varepsilon_2) \subset \bigcup_{k_1 \leq 0} \varepsilon_1^{k_1} (\mathcal{B} \cup \varepsilon_2 \mathcal{B}). \] 
As before, we then use \textit{3)} of Proposition \ref{propofcolmezdom2} to deduce that $C(1, \varepsilon_2) \cap \pi^{-1}C(1, \varepsilon_2)= \emptyset$. This allows us to conclude.
\end{proof}

\begin{remark}
We remark here that for some choices of $\pi$, $\varepsilon_1$ and $\varepsilon_2$ we have the stronger inclusion
\[ \pi^{-1} \mathcal{B} \subset \bigcup_{k_1=0}^1 \bigcup_{k_2=0}^{1} \varepsilon_1^{k_1} \varepsilon_2^{k_2} \mathcal{B}.  \]
In the next section, we need to divide into these two cases. At this point we include examples of how each case can look to aid the reader when considering our proofs.
\label{remarkinclusion}
\end{remark}

\subsection{Explicit calculations}

Let $ V = \langle \varepsilon_1, \varepsilon_{2} \rangle$, where $\varepsilon_1, \varepsilon_2$ are as chosen before and write $\varepsilon_3=\pi$. Before continuing we are required to choose an auxiliary prime $\lambda$ such that:
\begin{itemize}
    \item $\lambda$ is $\pi$-good for $\mathcal{B}$ and $\mathcal{D}_V$, where $\mathcal{D}_V$ is as defined in Proposition \ref{firstchangeofdom},
    \item $\lambda$ is good for $(\mathcal{D}_V, \mathcal{B}$).
\end{itemize}
In \cite{MR2420508} (after Definition 3.16) Dasgupta notes that given a Shintani domain $D$ all but finitely many prime ideals $\eta$ of $F$, with $\N \eta$ prime, are $\pi$-good for $D$. In particular Dasgupta notes that the set of such primes has Dirichlet density 1. Again in \cite{MR2420508} (after the proof of Theorem 5.3) Dasgupta notes that for any pair of Shintani domain $(D,D^\prime)$ all but finitely many prime ideals $\eta$ of $F$, with $\N \eta$ prime, are good for $D$.

It follows that there are an infinite number of primes $\lambda$ which satisfy the properties written above. Note that moving from a Shintai domain to a Colmez domain will not cause any issues here. Hence, such a choice of $\lambda$ is always possible. We fix this choice of $\lambda$ from now on. Proposition \ref{canchangedom} implies
\[ u_{\mathfrak{p}, \lambda}(\mathfrak{b},  \mathcal{B})= u_{\mathfrak{p}, \lambda}(\mathfrak{b},  \mathcal{D}_V). \]
By Proposition \ref{firstchangeofdom} we see that to prove Theorem \ref{thmforneq3} it only remains for us to show that, for any continuous homomorphism $g:F_\mathfrak{p}^\ast \rightarrow K$, such that $g$ is trivial on $E_+(\mathfrak{f})$, we have
\[ g(u_{\mathfrak{p}, \lambda}(\mathfrak{b},  \mathcal{B}))=  c_{g} \cap (\omega_{\mathfrak{f}, \mathfrak{b}, \lambda,V}^\mathfrak{p} \cap \vartheta_V^\prime). \]
Recall the definitions of the left and right side of the above equation from the start of $\S 6.1$. We will show the above equality by explicitly calculating each side. We begin by considering the right hand side. For $i=1, 2,3$ write, 
\[ \mathcal{B}_i \coloneqq  \bigcup_{\substack{\tau \in S_{3} \\ \tau(3)=i}} \overline{C}_{e_1}([\varepsilon_{\tau (1)} \mid \varepsilon_{\tau(2)} ]). \]
this was already defined for $i=1,2$ and note that $\mathcal{B}_3=\mathcal{B}$. We choose the following generator for $H_3(E_+(\mathfrak{f})_{\mathfrak{p}},\mathbb{Z})$,
\[ \vartheta_V^\prime =\sum_{\tau \in S_3} \sign(\tau)[\varepsilon_{\tau(1)} \mid \varepsilon_{\tau(2)} \mid \varepsilon_{\tau(3)}]\otimes 1. \]
This choice is stated by Spie\ss \ in Remark 2.1(c) of \cite{MR3201900}. We can now calculate
\[ \omega_{\mathfrak{f}, \mathfrak{b}, \lambda,V}^\mathfrak{p} \cap \vartheta^\prime_V = (-1)^{6} \sum_{i=1}^3 \sum_{\substack{ \tau \in S_{3} \\ \tau (3)=i}} \sign(\tau) \omega_{\mathfrak{f}, \mathfrak{b}, \lambda,V}^\mathfrak{p}([\varepsilon_{\tau(1)} \mid \varepsilon_{\tau(2)} ]) \otimes [\varepsilon_i]. \]
We recall the definition of $\omega_{\mathfrak{f}, \mathfrak{b}, \lambda,V}^\mathfrak{p}$ from the start of $\S 6.1$. Using that we have chosen $V$ and $\pi$ through Proposition \ref{propofcolmezdom2} we note that for $\tau \in S_{3}$ and a compact open $U \subseteq \mathcal{O}_\mathfrak{p}$, we have by definition that,
\begin{equation}
    \sign(\tau) \omega_{\mathfrak{f}, \mathfrak{b}, \lambda,V}^\mathfrak{p}([\varepsilon_{\tau(1)} \mid \varepsilon_{\tau(2)}]) = \zeta_{R, \lambda}(\mathfrak{b},\overline{C}_{e_1} ([\varepsilon_{\tau(1)} \mid  \varepsilon_{\tau(2)}]), U,0).
    \label{signcalc1}
\end{equation}

Recall that we can choose as a representative of $c_g$ the inhomogeneous $1$-cocycle $z_g=z_{\mathbbm{1}_{\pi\mathcal{O}_\mathfrak{p}},g} $, i.e. we take $f=\mathbbm{1}_{\pi\mathcal{O}_\mathfrak{p}}$ in Definition \ref{1cocycle}. One can easily compute, as is done by Dasgupta-Spie\ss \ in the proof of Proposition 4.6 in \cite{MR3968788}, that for $i=1,2$,
\begin{equation}
    \varepsilon_i^{-1} z_{g}(\varepsilon_i)=\mathbbm{1}_{\pi \mathcal{O}_\mathfrak{p}} \cdot g(\varepsilon_i), 
    \label{calc2}
\end{equation}
and
\begin{equation}
    \pi^{-1}z_{g}(\pi) = \mathbbm{1}_{\mathbb{O}}\cdot g +\mathbbm{1}_{\mathcal{O}_\mathfrak{p}} \cdot g(\pi) .
    \label{calc1}
\end{equation}

Returning to our main calculation, using (\ref{signcalc1}) we have, 
\[ c_g \cap (\omega_{\mathfrak{f}, \mathfrak{b}, \lambda, V}^\mathfrak{p} \cap \vartheta^\prime_V ) =  \sum_{i=1}^3  \sum_{\substack{ \tau \in S_{3} \\ \tau (3)=i}} \int_{F_\mathfrak{p}} z_{g}(\varepsilon_i)(x)d(\varepsilon_i \zeta_{R, \lambda}( \mathfrak{b},\overline{C}_{e_1}
([ \varepsilon_{\tau(1)} \mid \varepsilon_{\tau(2)} ]),x,0) ). \]
Applying (\ref{calc2}) and (\ref{calc1}) and piecing together the appropriate Shintani sets we further deduce, 
\begin{multline}
    c_g \cap (\omega_{\mathfrak{f}, \mathfrak{b}, \lambda,V}^\mathfrak{p} \cap \vartheta^\prime_V ) = \int_{\mathbb{O}} g(x)d(\zeta_{R,\lambda}(\mathfrak{b}, \mathcal{B},x,0)  ) + \int_{\mathcal{O}_\mathfrak{p}} g(\pi)d(\zeta_{R,\lambda}(\mathfrak{b}, \mathcal{B},x,0)  ) \\ + \sum_{i=1}^{2} \int_{\pi\mathcal{O}_\mathfrak{p}}  g(\varepsilon_i)d(\zeta_{R,\lambda}(\mathfrak{b}, \mathcal{B}_i,x,0)  ) .
    \label{calcstep1}
\end{multline}
Considering the first two terms on the right hand side of (\ref{calcstep1}) it is clear that 
\begin{multline*}
    \int_{\mathbb{O}} g(x)d(\zeta_{R,\lambda}(\mathfrak{b}, \mathcal{B},x,0)  ) + \int_{\mathcal{O}_\mathfrak{p}} g(\pi)d(\zeta_{R,\lambda}(\mathfrak{b}, \mathcal{B},x,0)  ) \\ = g \left( \pi^{\zeta_{R,\lambda}(\mathfrak{b}, \mathcal{B},\mathcal{O}_\mathfrak{p},0) }   \multint_{\mathbb{O}} x  d( \zeta_{R,\lambda}(\mathfrak{b}, \mathcal{B},x,0) )(x)    \right).
\end{multline*}
We now consider the sum on the right hand side of (\ref{calcstep1}). It is straight forward to see that
\[ \sum_{i=1}^{2} \int_{\pi\mathcal{O}_\mathfrak{p}}  g(\varepsilon_i)d(\zeta_{R,\lambda}(\mathfrak{b}, \mathcal{B}_i,x,0)  ) = g \left( \prod_{i=1}^{2} \varepsilon_i^{\zeta_{R,\lambda}(\mathfrak{b}, \mathcal{B}_i,\pi \mathcal{O}_\mathfrak{p},0)} \right). \]
Thus it only remains for us to prove the following equality
\[ g \left( \prod_{i=1}^{2} \varepsilon_i^{\zeta_{R,\lambda}(\mathfrak{b}, \mathcal{B}_i,\pi \mathcal{O}_\mathfrak{p},0)} \right) = g \left( \prod_{\epsilon \in V} \epsilon ^{\zeta_{R,\lambda}(\mathfrak{b}, \epsilon \mathcal{B} \cap \pi^{-1} \mathcal{B} ,\mathcal{O}_\mathfrak{p},0)} \right) . \]
By Proposition \ref{whereispiinverse} we have 
\[
    \prod_{\epsilon \in V} \epsilon ^{\zeta_{R,\lambda}(\mathfrak{b}, \epsilon \mathcal{B} \cap \pi^{-1} \mathcal{B} ,\mathcal{O}_\mathfrak{p},0)}  = \varepsilon_1^{\sum_{k_2=0}^{2} \zeta_{R,\lambda}(\mathfrak{b}, \varepsilon_1\varepsilon_2^{k_2} \mathcal{B} \cap \pi^{-1} \mathcal{B} ,\mathcal{O}_\mathfrak{p},0) } \varepsilon_2^{\sum_{k_2=1}^{2} \sum_{k_1=0}^1 k_2 \zeta_{R,\lambda}(\mathfrak{b}, \varepsilon_1^{k_1}\varepsilon_2^{k_2} \mathcal{B} \cap \pi^{-1} \mathcal{B} ,\mathcal{O}_\mathfrak{p},0) } .
\]
Thus it remains for us to show that the following two equalities hold.
\begin{align}
    \zeta_{R,\lambda}(\mathfrak{b}, \mathcal{B}_1,\pi \mathcal{O}_\mathfrak{p},0) &= \sum_{k_2=0}^{2} \zeta_{R,\lambda}(\mathfrak{b}, \varepsilon_1\varepsilon_2^{k_2} \mathcal{B} \cap \pi^{-1} \mathcal{B} ,\mathcal{O}_\mathfrak{p},0), \label{toshow1} \\
    \zeta_{R,\lambda}(\mathfrak{b}, \mathcal{B}_2,\pi \mathcal{O}_\mathfrak{p},0) &= \sum_{k_2=1}^{2} \sum_{k_1=0}^1 k_2 \zeta_{R,\lambda}(\mathfrak{b}, \varepsilon_1^{k_1}\varepsilon_2^{k_2} \mathcal{B} \cap \pi^{-1} \mathcal{B} ,\mathcal{O}_\mathfrak{p},0). \label{toshow2}
\end{align}
We begin by considering the left hand sides and note that for $i=1,2$ by Proposition \ref{changeofvariable}
\[ \zeta_{R,\lambda}(\mathfrak{b}, \mathcal{B}_i,\pi \mathcal{O}_\mathfrak{p},0) = \zeta_{R,\lambda}(\mathfrak{b}, \pi^{-1} \mathcal{B}_i, \mathcal{O}_\mathfrak{p},0). \]

It will be useful for our remaining calculations to make explicit the boundary cones that are contained in $\mathcal{B}$, $\mathcal{B}_1$ and $\mathcal{B}_2$. To achieve this we first define 
\[ \mathcal{B}^\prime = C(1) \cup C(1, \varepsilon_1) \cup C(1, \varepsilon_2) \cup C(1, \varepsilon_1 \varepsilon_2) \cup C(1, \varepsilon_1, \varepsilon_1 \varepsilon_2) \cup C(1, \varepsilon_2, \varepsilon_1 \varepsilon_2). \]
By Lemma \ref{changeofvariable} and the fact that $\mathcal{B}$ and $\mathcal{B}^\prime$ are equal up to translation of the boundary cones by $E_+(\mathfrak{f})$, we note that for any $k_1, k_2 \in \{ 0,1,2 \}$ we have
\[ \zeta_{R,\lambda}(\mathfrak{b}, \varepsilon_1^{k_1}\varepsilon_2^{k_2} \mathcal{B} \cap \pi^{-1} \mathcal{B} ,\mathcal{O}_\mathfrak{p},0) = \zeta_{R,\lambda}(\mathfrak{b}, \varepsilon_1^{k_1}\varepsilon_2^{k_2} \mathcal{B}^\prime \cap \pi^{-1} \mathcal{B}^\prime ,\mathcal{O}_\mathfrak{p},0). \]
Here we are also making use of the fact that in Proposition \ref{whereispiinverse} we made no assumptions about the boundary cones of $\mathcal{B}$. Thus, from now on we will assume that $\mathcal{B} = \mathcal{B}^\prime$. We now consider $\mathcal{B}_1$ and $\mathcal{B}_2$. For $a,b,c \in \{ 0,1 \}$ we define the Shintani sets
\begin{align*}
    \mathcal{B}_1^\prime(a,b) &= C(\pi^a )  \cup C(\pi^b,  \varepsilon_2 \pi^b) \cup C( 1,\pi) \cup C(1, \varepsilon_2 \pi ) \cup C(1, \varepsilon_2 ,  \varepsilon_2 \pi ) \cup C( 1, \pi , \varepsilon_2 \pi ), \\
    \mathcal{B}_2^\prime(a,b) &= C(\pi^a ) \cup C(\pi^b, \varepsilon_1 \pi^b ) \cup C( 1,\pi)  \cup C(1, \varepsilon_1 \pi ) \cup C(1, \varepsilon_1 ,  \varepsilon_1 \pi ) \cup C( 1, \pi , \varepsilon_1 \pi ).
\end{align*}
By the definition of $\mathcal{B}_i$, for $i=1,2$, there exists $a_i, b_i \in \{0,1\}$ such that $\mathcal{B}_i$ and $\mathcal{B}_i^\prime(a_i,b_i)$ are equal up to translation of the boundary cones by $E_+(\mathfrak{f})$. Thus, by Lemma \ref{changeofvariable} we have the equalities
\[ \zeta_{R,\lambda}(\mathfrak{b}, \pi^{-1} \mathcal{B}_1, \mathcal{O}_\mathfrak{p},0) =\zeta_{R,\lambda}(\mathfrak{b}, \pi^{-1} \mathcal{B}^\prime_1(a_1 ,b_1), \mathcal{O}_\mathfrak{p},0), \]
and 
\[ \zeta_{R,\lambda}(\mathfrak{b}, \pi^{-1} \mathcal{B}_2, \mathcal{O}_\mathfrak{p},0) = \zeta_{R,\lambda}(\mathfrak{b}, \pi^{-1} \mathcal{B}^\prime_2(a_2, b_2), \mathcal{O}_\mathfrak{p},0). \]
From this point on we will assume that $a_i=b_i=1$ for $i=1,2$ and write $\mathcal{B}_i=\mathcal{B}_i^\prime(1,1)$ for $i=1,2$. The proof of our main result in all other cases will follow with exactly the same ideas and the calculations are almost identical. Hence we fix the choices of $\mathcal{B}$, $\mathcal{B}_1$ and $\mathcal{B}_2$ we have made. Note that we can make the same choice of $\mathcal{B}$ in all cases. We now recall that from this point on we have assumed
\begin{align*}
    \mathcal{B} &= C(1) \cup C(1, \varepsilon_1) \cup C(1, \varepsilon_2) \cup C(1, \varepsilon_1 \varepsilon_2) \cup C(1, \varepsilon_1, \varepsilon_1 \varepsilon_2) \cup C(1, \varepsilon_2, \varepsilon_1 \varepsilon_2), \\
    \mathcal{B}_1 &= C(\pi )  \cup C(\pi,  \varepsilon_2 \pi) \cup C( 1,\pi) \cup C(1, \varepsilon_2 \pi ) \cup C(1, \varepsilon_2 ,  \varepsilon_2 \pi ) \cup C( 1, \pi , \varepsilon_2 \pi ), \\
    \mathcal{B}_2 &= C(\pi )  \cup C(\pi,  \varepsilon_1 \pi) \cup C( 1,\pi) \cup C(1, \varepsilon_1 \pi ) \cup C(1, \varepsilon_1 ,  \varepsilon_1 \pi ) \cup C( 1, \pi , \varepsilon_1 \pi ).
\end{align*}
With these choices we will now show that the equalities (\ref{toshow1}) and (\ref{toshow2}) hold. We begin with the following simple lemma.

\begin{lemma}
We have the following inclusions
\begin{align*}
    \pi^{-1} \mathcal{B}_1 & \subset \mathcal{B} \cup \varepsilon_2 \mathcal{B}, \\
    \pi^{-1} \mathcal{B}_2 & \subset \bigcup_{k_1=0}^1 \bigcup_{k_2=0}^{1} \varepsilon_1^{k_1} \varepsilon_2^{k_2} \mathcal{B} .
\end{align*}
\end{lemma}

\begin{proof}
We begin by considering $\mathcal{B}_1$. By definition we have that $\pi^{-1} \mathcal{B}_1$ is bounded by the cones
\[ C(1), C(\pi^{-1}), C(\varepsilon_2), C(\varepsilon_2 \pi^{-1}), C(1, \varepsilon_2), C(1, \pi^{-1}), C(\varepsilon_2, \varepsilon_2 \pi^{-1}) , C(\pi^{-1}, \varepsilon_2 \pi^{-1}).  \]
Note that not all of the above cones will be contained in $\pi^{-1}\mathcal{B}_1$. By the definition of $\mathcal{B}$ and the fact that $\pi^{-1} \in \mathcal{B}$ we see that all of the following Shintani cones are contained in $\mathcal{B} \cup \varepsilon_2 \mathcal{B}$,
\[ C(1), C(\pi^{-1}), C(\varepsilon_2), C(\varepsilon_2 \pi^{-1}), C(1, \varepsilon_2), C(1, \pi^{-1}), C(\varepsilon_2, \varepsilon_2 \pi^{-1}). \]
It remains for us to show that $C(\pi^{-1}, \varepsilon_2 \pi^{-1}) \subset \mathcal{B} \cup \varepsilon_2 \mathcal{B}$. Since $C(\pi^{-1}, \varepsilon_2 \pi^{-1})$ and $C(\varepsilon_1 \pi^{-1}, \varepsilon_1 \varepsilon_2 \pi^{-1})$ are boundary cones for $\pi^{-1}\mathcal{B}$, Proposition \ref{whereispiinverse} gives the inclusions
\begin{align*}
    C(\pi^{-1}, \varepsilon_2 \pi^{-1})  & \subset \bigcup_{k_1=0}^1 \bigcup_{k_2=0}^{2} \varepsilon_1^{k_1} \varepsilon_2^{k_2} \mathcal{B}, \\
     C(\varepsilon_1 \pi^{-1}, \varepsilon_1 \varepsilon_2 \pi^{-1}) & \subset \bigcup_{k_1=0}^1 \bigcup_{k_2=0}^{2} \varepsilon_1^{k_1} \varepsilon_2^{k_2} \mathcal{B}.
\end{align*}
These inclusions together imply that
\[ C(\pi^{-1}, \varepsilon_2 \pi^{-1})  \subset \bigcup_{k_2=0}^{2} \varepsilon_2^{k_2} \mathcal{B}. \]
If we write $\varphi_{(g_1,g_2)}(\pi^{-1}) = (a,b)$ then by the choices made in Lemma \ref{lemmafindpi-1} we see that $b<l$. Hence, by Corollary \ref{directionofsides}, the curve $\varphi_{(g_1,g_2)}(C(\pi^{-1}, \varepsilon_2 \pi^{-1})) = \varphi_{(g_1,g_2)}(\pi^{-1}) + \mathcal{C}_{2,l} $ lies strictly below the curve $(0,2l)+\mathcal{C}_{2,l}$, while still being contained in $\bigcup_{k_2=0}^{2} \varepsilon_2^{k_2} \mathcal{B}$. Hence we have $C(\pi^{-1}, \varepsilon_2 \pi^{-1}) \subset \mathcal{B} \cup \varepsilon_2 \mathcal{B}$. This gives us the result for $\mathcal{B}_1$.

The proof of the result for $\mathcal{B}_2$ is almost identical. As before we use Proposition \ref{whereispiinverse} to deal with the cone $C(\pi^{-1}, \varepsilon_1 \pi^{-1})$.
\end{proof}

Using the above lemma we deduce 
\[ \zeta_{R,\lambda}(\mathfrak{b}, \pi^{-1} \mathcal{B}_1, \mathcal{O}_\mathfrak{p},0) = \zeta_{R,\lambda}(\mathfrak{b}, (\pi^{-1} \mathcal{B}_1 \cap \mathcal{B}) \cup  \varepsilon_2^{-1} (\pi^{-1} \mathcal{B}_1 \cap \varepsilon_2 \mathcal{B}) , \mathcal{O}_\mathfrak{p},0) \]
and
\begin{multline*}
    \zeta_{R,\lambda}(\mathfrak{b}, \pi^{-1} \mathcal{B}_2, \mathcal{O}_\mathfrak{p},0) = \zeta_{R,\lambda}(\mathfrak{b}, (\pi^{-1} \mathcal{B}_2 \cap \mathcal{B}) \cup  \varepsilon_1^{-1} (\pi^{-1} \mathcal{B}_2 \cap \varepsilon_1 \mathcal{B}) , \mathcal{O}_\mathfrak{p},0) \\ + \zeta_{R, \lambda}(\mathfrak{b}, \pi^{-1} \mathcal{B}_2 \cap (\varepsilon_2 \mathcal{B} \cup \varepsilon_1 \varepsilon_2 \mathcal{B})  , \mathcal{O}_\mathfrak{p}, 0).
\end{multline*}
We now need to consider two possible cases. It is possible that the final zeta function in the sum above will be $0$. This will happen when, as noted in Remark \ref{remarkinclusion}, we have the stronger inclusion
\[ \pi^{-1} \mathcal{B} \subset \bigcup_{k_1=0}^1 \bigcup_{k_2=0}^{1} \varepsilon_1^{k_1} \varepsilon_2^{k_2} \mathcal{B},  \]
rather than that which is written in the statement of Proposition \ref{whereispiinverse}. We note that in this case the sums on the right hand sides of (\ref{toshow1}) and (\ref{toshow2}) become
\[ \sum_{k_2=0}^1 \zeta_{R,\lambda}(\mathfrak{b}, \varepsilon_1\varepsilon_2^{k_2} \mathcal{B} \cap \pi^{-1} \mathcal{B} ,\mathcal{O}_\mathfrak{p},0), \]
and
\[  \sum_{k_1=0}^1 \zeta_{R,\lambda}(\mathfrak{b}, \varepsilon_1^{k_1}\varepsilon_2 \mathcal{B} \cap \pi^{-1} \mathcal{B} ,\mathcal{O}_\mathfrak{p},0), \]
    respectively. In the following proposition we will need to divide the proof into two cases to deal with this possibility. In the case of the stronger inclusion the following proposition will complete the proof of our main result, Theorem \ref{thmforneq3}. We will refer to the case of the stronger inclusion as Case 1 and the other as Case 2. We now include 2 pictures showing how Case 1 and Case 2 can arise in the example from before by making different choices of $\pi$. Note that we can choose $\pi$ up to a factor of $E_+(\mathfrak{f})$. Both these diagrams are calculated making explicit choices, as before we will give more details on this in the appendix. In each of the diagrams the blue lines are boundary cones of the translates of $\mathcal{B}$ required in each case, and the red lines are the boundary cones of $\pi^{-1} \mathcal{B}$ for each choice of $\pi$. We note as before that although the image appears to show that some of the lines overlap, this does not happen. This only occurs in the diagram due to the fixed thickness of the lines.

\begin{figure}[h]
   \centering
    \begin{minipage}{0.5\textwidth}
       \centering
        \includegraphics[width=\textwidth]{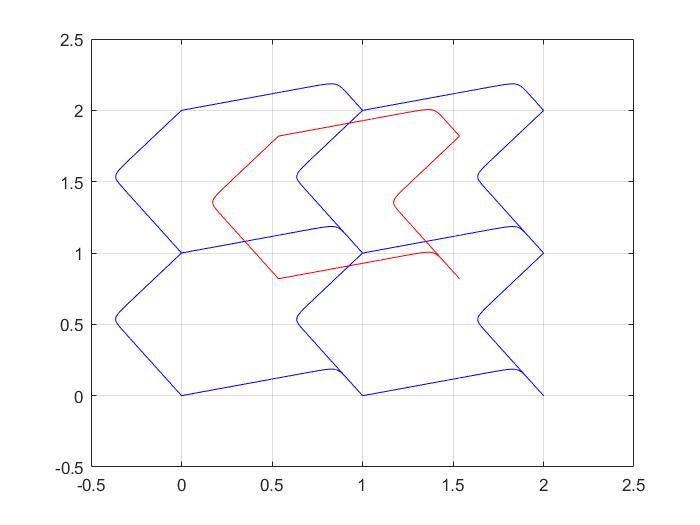} 
       \caption{Case 1}
       \label{fig:case1}
  \end{minipage}\hfill
   \begin{minipage}{0.5\textwidth}
      \centering
       \includegraphics[width=\textwidth]{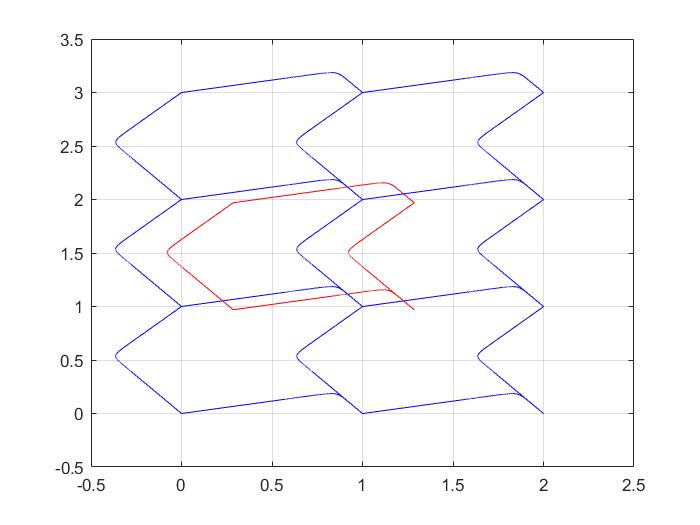} 
        \caption{Case 2}
       \label{fig:case2}
   \end{minipage}
\end{figure}

\begin{remark}
Figure \ref{fig:case2}, which concerns Case 2, has not been chosen by the methods outlines in Lemma \ref{lemmafindpi-1}. The reason for this is that the calculations necessary to draw the figures work very badly when working with subgroups $V \subset E_+(\mathfrak{f})$ of large index. Thus for the units we have chosen for the figures we would never choose an element $\pi$ so that we are in Case 2. However, to give the reader an idea of how this case would look we have found a choice of $\pi^{-1}$ that lies in the Colmez domin and is close to the region that Lemma \ref{lemmafindpi-1} gives to contain $\pi^{-1}$. Note that when working with subgroups $V \subset E_+(\mathfrak{f})$ of large index we are not able to guarantee that there always exists a choice of $\pi^{-1}$ in the region given by Lemma \ref{lemmafindpi-1} such that we land in Case 1. Hence we must continue to work with both cases.
\end{remark}

\begin{proposition}
In Case 1 we have
\begin{align*}
    \zeta_{R,\lambda}(\mathfrak{b}, (\pi^{-1} \mathcal{B}_1 \cap \mathcal{B}) \cup  \varepsilon_2^{-1} (\pi^{-1} \mathcal{B}_1 \cap \varepsilon_2 \mathcal{B}) , \mathcal{O}_\mathfrak{p},0) &= \sum_{k_2=0}^{1}\zeta_{R,\lambda}(\mathfrak{b},  \varepsilon_1 \varepsilon_2^{k_2}\mathcal{B} \cap \pi^{-1}\mathcal{B}, \mathcal{O}_\mathfrak{p}, 0),  \\
    \zeta_{R,\lambda}(\mathfrak{b}, (\pi^{-1} \mathcal{B}_2 \cap \mathcal{B}) \cup  \varepsilon_1^{-1} (\pi^{-1} \mathcal{B}_2 \cap \varepsilon_1 \mathcal{B}) , \mathcal{O}_\mathfrak{p},0) &= \sum_{k_1=0}^1  \zeta_{R,\lambda}(\mathfrak{b},  \varepsilon_1^{k_1} \varepsilon_2\mathcal{B} \cap \pi^{-1}\mathcal{B}, \mathcal{O}_\mathfrak{p}, 0).
\end{align*}
In Case 2 we have
\begin{align*}
    \zeta_{R,\lambda}(\mathfrak{b}, (\pi^{-1} \mathcal{B}_1 \cap \mathcal{B}) \cup  \varepsilon_2^{-1} (\pi^{-1} \mathcal{B}_1 \cap \varepsilon_2 \mathcal{B}) , \mathcal{O}_\mathfrak{p},0) &= \sum_{k_2=0}^{2}\zeta_{R,\lambda}(\mathfrak{b},  \varepsilon_1 \varepsilon_2^{k_2}\mathcal{B} \cap \pi^{-1}\mathcal{B}, \mathcal{O}_\mathfrak{p}, 0),  \\
    \zeta_{R,\lambda}(\mathfrak{b}, (\pi^{-1} \mathcal{B}_2 \cap \mathcal{B}) \cup  \varepsilon_1^{-1} (\pi^{-1} \mathcal{B}_2 \cap \varepsilon_1 \mathcal{B}) , \mathcal{O}_\mathfrak{p},0) &= \sum_{k_1=0}^1 \sum_{k_2=1}^{2} \zeta_{R,\lambda}(\mathfrak{b},  \varepsilon_1^{k_1} \varepsilon_2^{k_2}\mathcal{B} \cap \pi^{-1}\mathcal{B}, \mathcal{O}_\mathfrak{p}, 0).
\end{align*}
\end{proposition}

\begin{proof}
We first calculate 
\begin{align*}
    \sum_{k_2=0}^{2}\zeta_{R,\lambda}(\mathfrak{b},  \varepsilon_1 \varepsilon_2^{k_2}\mathcal{B} \cap \pi^{-1}\mathcal{B}, \mathcal{O}_\mathfrak{p}, 0) &= \zeta_{R,\lambda}(\mathfrak{b}, \bigcup_{k_2=0}^{2} \varepsilon_1^{-1} \varepsilon_2^{-k_2} ( \varepsilon_1 \varepsilon_2^{k_2}\mathcal{B} \cap \pi^{-1}\mathcal{B} ), \mathcal{O}_\mathfrak{p}, 0), \\
    \sum_{k_1=0}^1 \sum_{k_2=1}^{2} \zeta_{R,\lambda}(\mathfrak{b},  \varepsilon_1^{k_1} \varepsilon_2^{k_2}\mathcal{B} \cap \pi^{-1}\mathcal{B}, \mathcal{O}_\mathfrak{p}, 0) &= \zeta_{R,\lambda}(\mathfrak{b}, \bigcup_{k_1=0}^1 \bigcup_{k_2=1}^{2} \varepsilon_1^{-k_1} \varepsilon_2^{-k_2}( \varepsilon_1^{k_1} \varepsilon_2^{k_2}\mathcal{B} \cap \pi^{-1}\mathcal{B}), \mathcal{O}_\mathfrak{p}, 0).
\end{align*}
Thus if we can show the following equalities of Shintani sets we will be done
\begin{align}
    (\pi^{-1} \mathcal{B}_1 \cap \mathcal{B}) \cup  \varepsilon_2^{-1} (\pi^{-1} \mathcal{B}_1 \cap \varepsilon_2 \mathcal{B}) &= \bigcup_{k_2=0}^{2} \varepsilon_1^{-1} \varepsilon_2^{-k_2} ( \varepsilon_1 \varepsilon_2^{k_2}\mathcal{B} \cap \pi^{-1}\mathcal{B} ), \label{shintanieqn1} \\
    (\pi^{-1} \mathcal{B}_2 \cap \mathcal{B}) \cup  \varepsilon_1^{-1} (\pi^{-1} \mathcal{B}_2 \cap \varepsilon_1 \mathcal{B}) &= \bigcup_{k_1=0}^1 \bigcup_{k_2=1}^{2} \varepsilon_1^{-k_1} \varepsilon_2^{-k_2}( \varepsilon_1^{k_1} \varepsilon_2^{k_2}\mathcal{B} \cap \pi^{-1}\mathcal{B}).
    \label{shintanieqn2}
\end{align}
To show the above we will need to calculate each side in terms of explicit Shintani cones. We will begin by showing (\ref{shintanieqn1}). Recall we have defined the following 
\begin{align*}
    \mathcal{B} &= C(1) \cup C(1, \varepsilon_1) \cup C(1, \varepsilon_2) \cup C(1, \varepsilon_1 \varepsilon_2) \cup C(1, \varepsilon_1, \varepsilon_1 \varepsilon_2)  \cup C(1, \varepsilon_2, \varepsilon_1 \varepsilon_2), \\
    \pi^{-1} \mathcal{B}_1 &= C(1) \cup C(1, \varepsilon_2) \cup C(1, \pi^{-1}) \cup C(\pi^{-1}, \varepsilon_2 ) \cup C(\pi^{-1}, \varepsilon_2 , \varepsilon_2 \pi^{-1}) \cup C(1, \varepsilon_2,  \pi^{-1}).
\end{align*}
Let $\alpha \in C(\pi^{-1}, \varepsilon_2\pi^{-1}) \cap C(\varepsilon_2, \varepsilon_1 \varepsilon_2 ) $, we then have
\[ \pi^{-1} \mathcal{B}_1 \cap \mathcal{B} = C(1) \cup C(1, \varepsilon_2) \cup C(1, \pi^{-1}) \cup C(\pi^{-1} , \varepsilon_2 ) \cup C(\varepsilon_2, \pi^{-1}, \alpha) \cup C(1, \varepsilon_2, \pi^{-1}) \]
and
\[ \pi^{-1} \mathcal{B}_1 \cap \varepsilon_2 \mathcal{B} = C(\varepsilon_2, \alpha) \cup C(\varepsilon_2, \alpha,  \varepsilon_2 \pi^{-1}). \]
We can now explicitly write the left hand side of (\ref{shintanieqn1}). In particular, we have
\begin{multline*}
    (\pi^{-1} \mathcal{B}_1 \cap \mathcal{B}) \cup  \varepsilon_2^{-1} (\pi^{-1} \mathcal{B}_1 \cap \varepsilon_2 \mathcal{B}) \\ = C(1) \cup C(1, \varepsilon_2) \cup C(1, \pi^{-1}) \cup C(\pi^{-1} , \varepsilon_2 ) \cup C(\varepsilon_2, \pi^{-1}, \alpha) \cup C(1, \varepsilon_2, \pi^{-1}) \\ \cup C(1, \varepsilon_2^{-1} \alpha) \cup C(1, \varepsilon_2^{-1} \alpha, \pi^{-1} ).
\end{multline*}
We now consider the right hand side of (\ref{shintanieqn1}). Suppose that we are in Case 1, in this case the right hand side of (\ref{shintanieqn1}) becomes $\varepsilon_1^{-1} ( \varepsilon_1 \mathcal{B} \cap \pi^{-1}\mathcal{B} ) \cup \varepsilon_1^{-1} \varepsilon_2^{-1} ( \varepsilon_1 \varepsilon_2\mathcal{B} \cap \pi^{-1}\mathcal{B} )$. Let $\beta \in C(\varepsilon_1, \varepsilon_1 \varepsilon_2) \cap C(\pi^{-1}, \pi^{-1}\varepsilon_1)$, we can then calculate
\begin{multline*}
    \varepsilon_1 \mathcal{B} \cap \pi^{-1} \mathcal{B} = \\ C(\beta) \cup C(\beta, \varepsilon_1 \varepsilon_2) \cup C(\beta, \varepsilon_1 \pi^{-1}) \cup C(\varepsilon_1 \varepsilon_2, \varepsilon_1 \pi^{-1} ) \cup C( \varepsilon_1 \varepsilon_2 , \beta , \varepsilon_1 \pi^{-1}) \cup C( \varepsilon_1 \varepsilon_2, \varepsilon_1 \alpha , \varepsilon_1 \pi^{-1} ) 
\end{multline*}
and
\begin{multline*}
     \varepsilon_1\varepsilon_2 \mathcal{B} \cap \pi^{-1} \mathcal{B} = C(\varepsilon_1\varepsilon_2 ) \cup C(\varepsilon_1\varepsilon_2 , \varepsilon_2\beta) \cup C(\varepsilon_1\varepsilon_2 , \varepsilon_1 \alpha ) \cup C(\varepsilon_1\varepsilon_2 , \varepsilon_1\varepsilon_2  \pi^{-1}) \\ \cup C(\varepsilon_1\varepsilon_2 ,  \varepsilon_1 \alpha , \varepsilon_1\varepsilon_2 \pi^{-1}) \cup C(\varepsilon_1\varepsilon_2 ,  \varepsilon_2 \beta, \varepsilon_1\varepsilon_2 \pi^{-1}).
\end{multline*}
Using the fact that $\beta \in C(\varepsilon_1, \varepsilon_1 \varepsilon_2)$ we have
\begin{multline*}
    (\varepsilon_1 \mathcal{B} \cap \pi^{-1} \mathcal{B}) \cup \varepsilon_2^{-1}(\varepsilon_1\varepsilon_2 \mathcal{B} \cap \pi^{-1} \mathcal{B}) \\ = C(\varepsilon_1) \cup C(\varepsilon_1, \varepsilon_1 \varepsilon_2) \cup C(\varepsilon_1, \varepsilon_1 \pi^{-1}) \cup C(\varepsilon_1 \pi^{-1} ,\varepsilon_1 \varepsilon_2 ) \cup C(\varepsilon_1 \varepsilon_2, \varepsilon_1 \pi^{-1}, \varepsilon_1 \alpha) \cup C(\varepsilon_1, \varepsilon_1 \varepsilon_2, \varepsilon_1\pi^{-1}) \\ \cup C(\varepsilon_1, \varepsilon_1 \varepsilon_2^{-1} \alpha) \cup C(\varepsilon_1, \varepsilon_1 \varepsilon_2^{-1} \alpha, \varepsilon_1 \pi^{-1} ).
\end{multline*}
By multiplying the above by $\varepsilon_1^{-1}$, it is then clear that (\ref{shintanieqn1}) holds in Case 1. The proof of (\ref{shintanieqn1}) in Case 2 is very similar. The extra calculations which arise from being in Case 2 are very similar to those which we will deal with in our proof of (\ref{shintanieqn2}) in Case 2.

We now consider (\ref{shintanieqn2}). In Case 1 the proof is symmetric to the proof of (\ref{shintanieqn1}) in Case 1. So it only remains to show (\ref{shintanieqn2}) when we are in Case 2. Let $\alpha \in C(\pi^{-1}, \varepsilon_1 \pi^{-1}) \cap C(\varepsilon_2, \varepsilon_1 \varepsilon_2)$ and $ \beta \in C(\pi^{-1}, \varepsilon_1 \pi^{-1}) \cap C(\varepsilon_1 \varepsilon_2, \varepsilon_1^2 \varepsilon_2)$. Using similar calculations as before we deduce
\begin{multline*}
    (\pi^{-1} \mathcal{B}_2 \cap \mathcal{B}) \cup  \varepsilon_1^{-1} (\pi^{-1} \mathcal{B}_2 \cap \varepsilon_1 \mathcal{B}) \\ = C(1) \cup C(1, \varepsilon_2) \cup C(1, \varepsilon_1) \cup C(1, \varepsilon_1^{-1} \beta) \cup C(1, \pi^{-1}) \cup C(\pi^{-1}, \varepsilon_1) \cup C(\pi^{-1}, \varepsilon_1 \varepsilon_2) \\ \cup C(1, \varepsilon_2, \varepsilon_1^{-1} \beta ) \cup C(1, \pi^{-1}, \varepsilon_1^{-1} \beta) \cup C(1, \pi^{-1}, \varepsilon_1) \cup C(\varepsilon_1, \pi^{-1}, \varepsilon_1 \varepsilon_2) \cup C(\pi^{-1}, \alpha, \varepsilon_1 \varepsilon_2).
\end{multline*}
We are able to calculate that the same is also true for $\bigcup_{k_1=0}^1 \bigcup_{k_2=1}^{2} \varepsilon_1^{-k_1} \varepsilon_2^{-k_2}( \varepsilon_1^{k_1} \varepsilon_2^{k_2}\mathcal{B}^\prime \cap \pi^{-1}\mathcal{B}^\prime)$ and thus we complete the proof.

\end{proof}

This final proposition completes the proof our main result, Theorem \ref{thmforneq3}.

\begin{proposition}
If we are in Case 2 then,
\[ \zeta_{R, \lambda}(\mathfrak{b}, (\varepsilon_2 \mathcal{B} \cup \varepsilon_1 \varepsilon_2 \mathcal{B}) \cap \mathcal{B}_2, \mathcal{O}_\mathfrak{p}, 0) = \zeta_{R, \lambda}(\mathfrak{b}, (\varepsilon_1 \varepsilon_2\mathcal{B} \cup \varepsilon_2^2 \mathcal{B}) \cap \pi^{-1}\mathcal{B}, \mathcal{O}_{\mathfrak{p}}, 0). \]
\end{proposition}

\begin{proof}
Using Lemma \ref{changeofvariable} it is enough to show the following equality of Shintani sets
\[ (\varepsilon_2 \mathcal{B} \cup \varepsilon_1 \varepsilon_2 \mathcal{B}) \cap \mathcal{B}_2 = \varepsilon_2^{-1} ((\varepsilon_1 \varepsilon_2\mathcal{B} \cup \varepsilon_2^2 \mathcal{B}) \cap \pi^{-1}\mathcal{B}). \]
Again letting $\alpha \in C(\pi^{-1}, \varepsilon_1 \pi^{-1}) \cap C(\varepsilon_2, \varepsilon_1 \varepsilon_2)$ and $ \beta \in C(\pi^{-1}, \varepsilon_1 \pi^{-1}) \cap C(\varepsilon_1 \varepsilon_2, \varepsilon_1^2 \varepsilon_2)$ we are able to calculate that each side of the above equation is equal to
\[ C(\varepsilon_1 \varepsilon_2) \cup C(\alpha , \varepsilon_1 \varepsilon_2 ) \cup C(\varepsilon_1 \varepsilon_2, \beta)  \cup C(\alpha, \varepsilon_1 \varepsilon_2 , \beta ).   \]
This concludes the result.
\end{proof}

We end this section by proving Theorem \ref{thmforequality}. The key step is to note that if we replace $g$ by $\text{id}:F_\mathfrak{p}^\ast \rightarrow F_\mathfrak{p}^\ast$ in Proposition \ref{propofcolmezdom2} then we see that if we can show 
\[  u_{\mathfrak{p}, \lambda}(\mathfrak{b},  \mathcal{D}_V)=  c_{\text{id}} \cap (\omega_{\mathfrak{f}, \mathfrak{b}, \lambda,V}^\mathfrak{p} \cap \vartheta_V^\prime), \]
then we have
\[ u_{\mathfrak{p}, \lambda}(\mathfrak{b},  \mathcal{D})= \gamma_{[E_+(\mathfrak{f}):V]} ( c_{\text{id}} \cap (\omega_{\mathfrak{f}, \mathfrak{b}, \lambda}^\mathfrak{p} \cap \vartheta^\prime)). \]
Where $\gamma_{[E_+(\mathfrak{f}):V]}$ is a root of unity of order $[E_+(\mathfrak{f}):V]$. To prove Theorem \ref{thmforequality} it is thus enough for us to find two free subgroups $V, V^\prime \subseteq E_+(\mathfrak{f}) $ such that they are small enough to use in our work for Theorem \ref{thmforneq3} and such that $\gcd( [E_+(\mathfrak{f}):V] , [E_+(\mathfrak{f}):V^\prime] )=1$.

\begin{proof}[Proof of Theorem \ref{thmforequality}]
When we choose $g_1$ and $g_2$ we do so such that $\text{Log}( g_i)  \in B(l_i(M_1, r$) where $r$ and $M_1$ are as we write after Lemma \ref{fixgi}. Note that there is no upper bound on these choices, it is therefor clear that if we allow $r$ and $M_1$ to be large enough we can choose $g_1, g_2$ and  $g_1^\prime , g_2^\prime$ such that
\begin{itemize}
    \item $\langle g_1, g_2 \rangle$ and $\langle g_1^\prime, g_2^\prime \rangle$ are free of rank $2$,
    \item $g_1, g_2$ and  $g_1^\prime , g_2^\prime$ satisfy the properties of Lemma \ref{fixgi} and
    \item $[E_+(\mathfrak{f}): \langle g_1, g_2 \rangle ]  $ and $[E_+(\mathfrak{f}): \langle g_1^\prime, g_2^\prime \rangle ]$ are coprime.
\end{itemize}
Next we raise the $g_i$ by a large power $l$ in Corollary \ref{directionofsides} and Lemma \ref{lemmafindpi-1}. Again the only condidion on $l$ is that it is greater than a fixed lower bound, hence we can choose $l$ and $l^\prime$ such that they are coprime to each other and to \[ [E_+(\mathfrak{f}): \langle g_1, g_2 \rangle ] [E_+(\mathfrak{f}): \langle g_1^\prime, g_2^\prime \rangle ]. \]
We then get $V= \langle g_1^l , g_2^l \rangle$ and $V^\prime =\langle (g_1^\prime )^{l^\prime}, (g_2^\prime )^{l^\prime}\rangle $. Following our work for Theorem \ref{thmforneq3} we then see that 
\[  u_{\mathfrak{p}, \lambda}(\mathfrak{b},  \mathcal{D}_V)=   c_{\text{id}} \cap (\omega_{\mathfrak{f}, \mathfrak{b}, \lambda,V}^\mathfrak{p} \cap \vartheta_V^\prime) \quad \text{and} \quad  u_{\mathfrak{p}, \lambda}(\mathfrak{b},  \mathcal{D}_{V^\prime})=    c_{\text{id}} \cap (\omega_{\mathfrak{f}, \mathfrak{b}, \lambda,V^\prime}^\mathfrak{p} \cap \vartheta_{V^\prime}^\prime). \]
Hence,
\[  u_{\mathfrak{p}, \lambda}(\mathfrak{b},  \mathcal{D})= \gamma_{[E_+(\mathfrak{f}):V]} ( c_{\text{id}} \cap (\omega_{\mathfrak{f}, \mathfrak{b}, \lambda}^\mathfrak{p} \cap \vartheta^\prime)), \]
and 
\[  u_{\mathfrak{p}, \lambda}(\mathfrak{b},  \mathcal{D})= \gamma_{[E_+(\mathfrak{f}):V^\prime]} (  c_{\text{id}} \cap (\omega_{\mathfrak{f}, \mathfrak{b}, \lambda}^\mathfrak{p} \cap \vartheta^\prime)). \]
In the above, $\gamma_{[E_+(\mathfrak{f}):V]}$ is a root of order $[E_+(\mathfrak{f}):V]$ and $\gamma_{[E_+(\mathfrak{f}):V^\prime]}$ is a root of order $[E_+(\mathfrak{f}):V^\prime]$. Our choice of $V$ and $V^\prime$ gives that $\gcd( [E_+(\mathfrak{f}):V] , [E_+(\mathfrak{f}):V^\prime] )=1$. Thus $\gamma_{[E_+(\mathfrak{f}):V]}=\gamma_{[E_+(\mathfrak{f}):V^\prime]}=1$ and so we get the result.
\end{proof}

\appendix 
\section{Appendix: Translating Shintani domains}

Overcoming the lack of a nice translation property for Shintani domains in \S 6.2, is the main work of this paper. In this section, we first provide an explicit counterexample which shows why this work is necessary. We then show the calculations which give rise to the figures. These figures demonstrate our method to overcome this counterexample, namely, Figure \ref{fig:colmezdom}, Figure \ref{fig:case1},  and Figure \ref{fig:case2}. We begin by finding a counterexample to the following statement of Tsosie in \cite{tsosie2018compatibility}. The statement below is given for $F$ of any degree $n>1$. We will provide a counterexample with $F$ a cubic field as this is the case we work with in this paper.

\begin{statement}
Let $V$ be a finite index subgroup of $E_+(\mathfrak{f})$ and let $\epsilon_1, \dots , \epsilon_{n-1}$ be a $\mathbb{Z}$-basis for $V$. Further, let $\mathcal{D}$ be a fundamental domain for the action of $V$ on $\mathbb{R}^n_+$ and $\pi^{-1} \in \mathcal{D}$, then for $\epsilon = \prod_{i=1}^{n-1} \epsilon_i^{m_i}$,
\[ \epsilon \mathcal{D} \cap \pi^{-1} \mathcal{D} = \emptyset \]
unless $m_i \in \{ 0,1 \}$, $1 \leq i \leq n-1$.
\label{STconj}
\end{statement}

We note that in general there appears to be no bounds that can be put on the set which the $m_i$'s are allowed to be in to make this statement hold. However, we do not provide explicit evidence for this here.

\begin{remark}
It is straightforward to show that this statement holds when $F$ is of degree 2. It is for this reason that Dasgupta-Spie\ss's proof for the consistency of Conjecture \ref{conj3.21} and Conjecture \ref{conjDS}, in the case $F$ is of degree $2$, is much shorter.
\end{remark}

The computations used to find our counterexample below are done using Magma. Let $F$ be the number field with defining polynomial $2x^3-4x^2-x+1$ over $\mathbb{Q}$. $F$ is then a totally real number field of degree $3$. We define 
\[H=F(\sqrt{-2}),\]
$H$ is then totally complex. It is also a degree $2$ extension of $F$ so $H$ is a CM-number field. We note that the extension $H/F$ is abelian. Now, choose $y \in F$ such that we can write
\[ F=\mathbb{Q}(y). \]
Let $\mathfrak{f}$ be the conductor of $H/F$. We calculate, as the generators of $E_+(\mathfrak{f})$, the elements $g_1 = -96y^2+152y+113$ and $g_2 = 160y^2 +32y -31$, i.e., we have
\[ \langle  -96y^2+152y+113 \ , \ 160y^2 +32y -31  \rangle = E_+(\mathfrak{f}) . \]
We choose as our rational prime $p=113$. We make this choice as there are two primes of $F$ above $113$ and both of them split completely in $H$. We choose $\mathfrak{p} \mid p$, a prime ideal of $F$ that splits completely in $H$. We find that the order of $\mathfrak{p}$ in $G_\mathfrak{f}$ is $2$. We choose an element $\pi$ to satisfy:
\begin{itemize}
    \item $\pi$ is totally positive,
    \item $\pi \equiv 1 \pmod{\mathfrak{f}}$,
    \item $(\pi)=\mathfrak{p}^2$,
    \item $\pi^{-1} \in \overline{C}_{e_1}([g_1 \mid g_2]) \cup \overline{C}_{e_1}([g_2 \mid g_1])$.
\end{itemize}
In particular, we choose $\pi = 192y^2-488y+177$. Let $\mathcal{D}=\overline{C}_{e_1}([g_1 \mid g_2]) \cup \overline{C}_{e_1}([g_2 \mid g_1])$ and note that this is a Shintani domain. With these choices, we calculate that $\pi^{-1} \mathcal{D} \cap g_1g_2^{-1}\mathcal{D} \neq \emptyset$ and $\pi^{-1} \mathcal{D} \cap g_2^{-1}\mathcal{D} \neq \emptyset$. This completes our counterexample to Statement \ref{STconj}. Furthermore, the curved nature of the domains, as illustrated further with the picture below, gives us a good reason as to why results bounding where $\pi^{-1}\mathcal{D}$ is contained should not be possible without considerable work.

To make our example clearer, we include below a plot of $\mathcal{D} \cup g_1 \mathcal{D} \cup g_2 \mathcal{D} \cup g_1g_2 \mathcal{D}$ (in blue) and $\pi^{-1}\mathcal{D}$ (in red) under the map $\varphi_{(g_1, g_2)}$. This plot is drawn using MATLAB. Notice that the boundary of $\pi^{-1}\mathcal{D}$ falls outside that of $\mathcal{D} \cup g_1 \mathcal{D} \cup g_2 \mathcal{D} \cup g_1g_2 \mathcal{D}$. As we remarked with the other diagrams, although the image appears to show that some of the lines overlap, this does not happen. This only appears in the diagram due to the fixed thickness of the lines.

\begin{figure}[h]
    \centering
    \includegraphics[scale=0.7]{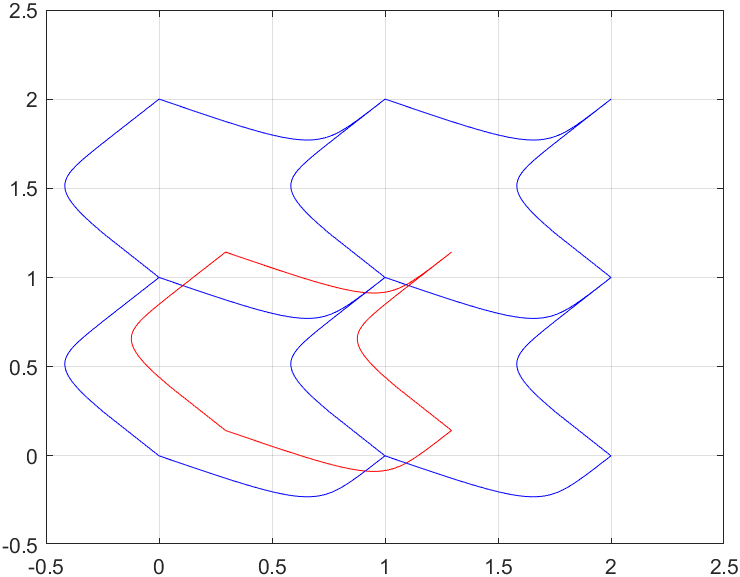}
    \caption{The counter-example}
    \label{fig:counterexample}
\end{figure}

We now make note of the calculations we made to obtain Figure \ref{fig:colmezdom}, Figure \ref{fig:case1}, and Figure \ref{fig:case2}. We continue to hold all of the choices which have been made so far in this appendix. We define
\[ \varepsilon_1=g_1^{-3} g_2^4 \quad  \text{and} \quad  \varepsilon_2=g_1^{-5}. \]
These choices are found using Magma so that $\varepsilon_1$ and $\varepsilon_2$ satisfy the conditions in Lemma \ref{fixgi}. We find that when considering Corollary \ref{directionofsides}, we can choose $l=1$ to satisfy the conditions given, i.e., $\varepsilon_1$ and $\varepsilon_2$ are already good enough to obtain Corollary \ref{directionofsides}. Using MATLAB, we plot Figure \ref{fig:colmezdom}. We define 
\[ \pi_1= g_1^{-6} g_2^2 \pi \quad \text{and} \quad \pi_2= g_1^{-6} g_2 \pi , \]
where $\pi$ is as we defined before.fa Using $\pi_1$ as our choice of $\pi$, and using MATLAB, we plot Figure \ref{fig:case1} which shows Case 1. Similarly, using $\pi_2$ as our choice of $\pi$, Figure \ref{fig:case2} shows Case 2.

\addcontentsline{toc}{section}{References}

\bibliography{bib}
\bibliographystyle{plain}

\end{document}